\documentclass[reqno]{amsart} 
 
\usepackage{hyperref}
\usepackage{amsmath}
\usepackage{amssymb}
\usepackage{amsfonts}
\usepackage{amsthm}
\usepackage{latexsym}
\usepackage{esint}
\usepackage{mathtools}
\usepackage{mathrsfs}
\usepackage{graphicx}
\usepackage{wrapfig}
\usepackage{bbm}
\usepackage{color}
\usepackage{bm}
\usepackage[top=1in, bottom=1.25in, left=1.10in, right=1.10in]{geometry}
\usepackage{todonotes}
\usepackage{stackengine}
\allowdisplaybreaks

\newcommand\ocirc[1]{\ensurestackMath{\stackon[1pt]{#1}{\mkern2mu\circ}}}

\newtheorem{theorem}{Theorem}[section]

\newtheorem{proposition}[theorem]{Proposition}
\newtheorem{corollary}[theorem]{Corollary}

\theoremstyle{definition}
\newtheorem{definition}[theorem]{Definition}

\theoremstyle{remark}
\newtheorem{remark}[theorem]{Remark}

\numberwithin{equation}{section}

\newcommand{\bw}{\mathbf{w}}
\newcommand{\bx}{\mathbf{x}}
\newcommand{\by}{\mathbf{y}}

\newcommand{\bu}{\mathbf{u}}
\newcommand{\bfu}{\mathbf{u}}

\newcommand{\bff}{\mathbf{f}}

\newcommand{\bv}{\mathbf{v}}
\newcommand{\bn}{\mathbf{n}}

\newcommand{\dy}{\, \mathrm{d}\by}
\newcommand{\dd}{\,\mathrm{d}}

\newcommand{\dx}{\, \mathrm{d} \mathbf{x}}

\newcommand{\dt}{\, \mathrm{d}t}

\newcommand{\Div}{\mathrm{div}_{\mathbf{x}}}
\newcommand{\divx}{\mathrm{div}_{\mathbf{x}}}

\newcommand{\delx}{\Delta_{\mathbf{x}}}
\newcommand{\dely}{\Delta_{\mathbf{y}}}

\newcommand{\nabx}{\nabla_{\mathbf{x}}}
\newcommand{\naby}{\nabla_{\mathbf{y}}}

\newcommand{\Delx}{\Delta_{\mathbf{x}}}
\newcommand{\Dely}{\Delta_{\mathbf{y}}}

\newcommand{\bT}{\mathbb{T}}
\newcommand{\bD}{\mathbb{D}}
\newcommand{\bW}{\mathbb{W}}
\newcommand{\R}{\mathbb{R}}

\newcommand{\Oeta}{\Omega_{\eta}}
\newcommand{\Ozeta}{\Omega_{\zeta}}

\newcommand{\Testzeta}{{\mathscr{F}}^{\Div}_{\eta}}

\begin{document}


\title[Polymeric fluid-structure interaction]{Strong solution for polymeric fluid-structure interaction with small initial acceleration}

%
%
\author{Prince Romeo Mensah}
\address{Institut f\"ur Mathematik,
Technische Universit\"at Clausthal,
Erzstra{\ss}e 1,
38678 Clausthal-Zellerfeld, Germany
\\ and
\\
Faculty of Mathematics, University of Duisburg-Essen, Thea-Leymann-Stra{\ss}e 9, D-45127 Essen, Germany}
\email{prince.mensah@uni-due.de
  \\ orcid ID: 0000-0003-4086-2708
}

\subjclass[2020]{76D03; 74F10; 35Q30; 35Q84; 82D60}

\date{\today}


\keywords{Navier-Stokes equation, Oldroyd-B, Fluid-Structure interaction, Polymeric fluid, Maximal regularity}

\begin{abstract}  
 We consider the problem of a 3D-3D-2D mutually coupled solute-solvent-structure  three-states system. This describes the interaction of a flexible structure  with a polymeric fluid of classical Oldroyd-B type without centre-of-mass diffusion. We construct a unique higher-order regularity notion of a strong solution for the system by decoupling the solute from the solvent-structure subsystem, solving the decoupled system individually, and gluing the solutions through a fixed-point argument. As a requirement for the construction, we rely on  a maximal regularity result for the Stokes problem on moving domains with non-trivial boundary conditions; a result that is also of independent interest.
%
\end{abstract}

\maketitle

\section{Introduction}
A fluid-structure interaction problem usually involves the coupling of the Navier--Stokes equation or the Euler equation with a purely elastic or viscoelastic beam, shell or wave equation. The rigorous analysis of such problems has seen profound contributions in the last two decades \cite{breit2022regularity, BMSS, chambolle2005existence, coutand2006interaction, cheng2010interaction,  GH, GHL, lengeler2014weak, Leq, maity2020maximal, muha2013existence, muha2014existence, beirao2004existence}  and will generally involve the analysis of either distributional solutions or solutions that satisfy the aforementioned coupled system pointwise. Furthermore, earlier results concentrated on idealized settings with flat reference configurations that allows for easier analysis. The advancement in analytical techniques has generated a recent resurgence into the analysis of more realistic fluid-structure interaction problems with arbitrary reference configurations \cite{breit2022regularity, breit2024martingale, BMSS, lengeler2011globale, lengeler2014weak, muha2019existence}. To present a prototype for such a problem, one first consider a  deformed spacetime cylinder
\begin{align*}
I\times\Oeta:=\bigcup_{t\in I}\{t\}\times\Oeta \in \mathbb{R}^{1+3}
\end{align*}
where $I:=(0,T)$ is a fixed time horizon and $\Oeta:={\Omega_{\eta(t)}}$ is a time-dependent flexible spatial domain whose closure is obtained through the parametrization of the boundary $\omega\subset \mathbb{R}^2$ of some reference configuration $\Omega\subset \mathbb{R}^3$.    If the structure is viscoelastic and one wishes to incorporate any kind of force (e.g. conservative and non-conservative forces) in the solvent, a physically sound and analytically well-posed model for such problem is the following system of equations
\begin{equation}
\begin{aligned}
\divx \bu=0, \label{x1} 
\\
\partial_t \bu  + (\mathbf{u}\cdot \nabx)\mathbf{u} 
= 
\delx \bu -\nabx p
+\bff
+
\divx   \bT,  
\\
\partial_t^2 \eta - \partial_t\dely \eta + \dely^2 \eta = g - ( \mathbb{S}\bn_\eta )\circ \bm{\varphi}_\eta\cdot\bn \,\vert\det(\naby\bm{\varphi}_\eta)\vert,  
\end{aligned}
\end{equation}
defined in $I\times\Oeta\subset \mathbb R^{1+3}$ (with the last equation defined on $I\times\omega\subset \mathbb R^{1+2}$).  Here
\begin{align*}
&\det(\naby\bm{\varphi}_\eta):=\partial_{y_1} \bm{\varphi}_{\eta}  \times \partial_{y_2} \bm{\varphi}_{\eta},
\\&\mathbb{D}(\nabx \bu)=\tfrac{1}{2}\big(\nabx \bu+(\nabx \bu)^\top \big),
\\&
\mathbb{S}=  2\mathbb{D}(\nabx \bu)  -p\mathbb{I}+ \bT,
\end{align*} 
the vector $\bn_\eta$ is the normal at $\partial\Oeta$, $\mathbb{I}$ is the $3\times 3$ identity matrix and where all underlining dimensionless parameters have been set to one for simplicity. 
The system \eqref{x1} is complemented by the following initial and boundary conditions
\begin{equation}
\begin{aligned}\label{x2}
&\eta(0,\cdot)=\eta_0(\cdot), \qquad\partial_t\eta(0,\cdot)=\eta_\star(\cdot) & \text{in }\omega,  
\\  
&\bu(0,\cdot)=\bu_0(\cdot) & \text{in }\Omega_{\eta_0}.
\\ 
&\bu\circ\bm{\varphi}_\eta=(\partial_t\eta)\bn & \text{on }I\times \omega.
\end{aligned}
\end{equation}
The scalar functions
$g:(t, \by)\in I \times \omega \mapsto   \eta(t,\by)\in \mathbb{R}$, the vector field $\bff:(t, \mathbf{x})\in I \times \Oeta \mapsto  \mathbf{u}(t, \mathbf{x}) \in \mathbb{R}^3$ and the tensor field $\bT :(t, \mathbf{x} )\in I \times \Oeta  \mapsto \bT (t, \mathbf{x} ) \in \mathbb{R}^{3\times3}$ in \eqref{x1} are known given forces whereas the structure displacement function
$\eta:(t, \by)\in I \times \omega \mapsto   \eta(t,\by)\in \mathbb{R}$, the solvent velocity field $\mathbf{u}:(t, \mathbf{x})\in I \times \Oeta \mapsto  \mathbf{u}(t, \mathbf{x}) \in \mathbb{R}^3$ and the pressure function $p:(t, \mathbf{x})\in I \times \Oeta \mapsto  p(t, \mathbf{x}) \in \mathbb{R}$ are the unknowns of the problem. Several kinds of solutions can be constructed for \eqref{x1}-\eqref{x2} including distributional  solutions (or weak solutions) \cite{breit2021incompressible} and solutions that satisfies \eqref{x1}-\eqref{x2} pointwise almost everywhere (or strong solutions)  but with different regularity class \cite{ breit2023existence, BMSS}. For the purpose of this work, and to better present the main subject of this work later, we are interested in the high regularity class of strong solutions whose precise definition is: 
\begin{definition}[Strong solution]
\label{def:strongSolutionAlone}
Let $(\bff, g, \eta_0, \eta_\star, \bu_0, \bT)$ be a dataset that satisfies
\begin{equation}
\begin{aligned}
\label{datasetAlone}
&\bff \in W^{1,2}(I;L^{2}(\Oeta))\cap L^{2}(I;W^{2,2} (\Oeta)),
\qquad 
\bff(0)\in W^{1,2}(\Omega_{\eta_0}),
\\&g\in L^2(I;W^{2,2}(\omega)) \cap  W^{1,2}(I;W^{1,2}(\omega)) ,
\\&
\eta_0 \in W^{5,2}(\omega) \text{ with } \Vert \eta_0 \Vert_{L^\infty( \omega)} < L, \quad \eta_\star \in W^{3,2}(\omega),
\\&\bu_0 \in W^{3,2}_{\divx}(\Omega_{\eta_0} )\text{ is such that }\bu_0 \circ \bm{\varphi}_{\eta_0} =\eta_\star \bn \text{ on } \omega,
\\& 
\bT \in   W^{1,2} \big(I; W^{2,2}(\Oeta ) \big)\cap L^2\big(I;W^{3,2}(\Oeta )  \big). 
\end{aligned}
\end{equation}
We say that
$( \eta, \bu,  p )$
is a \textit{strong solution}  of  \eqref{x1}--\eqref{x2} with  dataset $(\bff, g, \eta_0, \eta_\star, \bu_0, \bT)$ if the following holds:
\begin{itemize}
\item[(a)]  the structure displacement $\eta$ is such that $
\Vert \eta \Vert_{L^\infty(I \times \omega)} <L$ and
\begin{align*}
\eta \in W^{1,\infty}\big(I;W^{3,2}(\omega)  \big)\cap L^{\infty}\big(I;W^{5,2}(\omega)  \big) \cap  W^{3,2}\big(I; L^{2}(\omega)  \big)\cap  W^{2,2}\big(I;W^{2,2}(\omega)  \big) \cap  L^2\big(I;W^{6,2}(\omega)  \big);
\end{align*}
\item[(b)] the velocity $\bu$ is such that $\bu  \circ \bm{\varphi}_{\eta} =(\partial_t\eta)\bn$ on $I\times \omega$ and
\begin{align*} 
\bu\in  W^{2,2} \big(I; L^{2}_{\divx}(\Oeta ) \big)\cap W^{1,2} \big(I; W^{2,2}(\Oeta ) \big)\cap L^2\big(I;W^{4,2}(\Oeta)  \big);
\end{align*}
\item[(c)] the pressure $p$ is such that 
\begin{align*}
p\in W^{1,2} \big(I; W^{1,2}(\Oeta ) \big)\cap L^2\big(I;W^{3,2}(\Oeta)  \big);
\end{align*}
\item[(d)] the equations \eqref{x1} are satisfied a.e. in spacetime with $\eta(0)=\eta_0$ and $\partial_t\eta=\eta_\star$ a.e. in $\omega$ as well as $\bfu(0)=\bfu_0$ a.e. in $\Omega_{\eta_0}$.
\end{itemize}
\end{definition}
\noindent 
With this notion of a solution, we have the following result on the existence of a unique local strong solution for the solvent-structure subproblem.
\begin{theorem}\label{thm:BM}
Suppose that the dataset $(\bff, g, \eta_0, \eta_\star, \bu_0, \bT)$ satisfies \eqref{datasetAlone}.
There is a time $T_*>0$ such that \eqref{x1}--\eqref{x2} admits a unique strong solution $(\eta,\bu,p)$, in the sense of Definition \ref{def:strongSolutionAlone}.
\end{theorem}
The existence of a local-in-time strong solution to \eqref{x1}--\eqref{x2} in the sense of Definition \ref{def:strongSolutionAlone} has recently been shown in \cite[(3.7) \& Theorem 3.2]{breit2023existence}, see also \cite[Proposition 3.5]{breit2023existence}. This follows the now standard method where: 
 \begin{itemize}
 \item  the solvent-structure system is transformed to its reference domain.
\item The resulting system on the reference domain is then linearised and estimates for the
linearised system are obtained.
\item One then construct a contraction map for the linearised problem (by choosing the end time small
enough) leading to a unique local solution to the system on its original/actual domain.
 \end{itemize}

Our main interest in this work, however, is to study a mutually coupled three-states solute-solvent-structure problem where the solute is described by the dynamic evolution of the tensor $\bT$. This describes the fully macroscopic problem of a dilute 3-dimensional polymer fluid of Oldroyd-B type interacting with a 2-dimensional viscoelastic shell. This is part of a recent program to understand the macroscopic evolution of such a solute-solvent-structure (SSS) three matter state problem and this is the first attempt at the 3D-3D-2D setting as the others tackled the 2D-2D-1D problem. Also, compared to \cite{mensah2023weak}, the solute subcomponent in this current work is purely elastic without any centre-of-mass diffusion and thus, poses a significant analytical challenge.
 More precisely, our goal is to find a structure displacement function $\eta:(t, \by)\in I \times \omega \mapsto   \eta(t,\by)\in \mathbb{R}$, a fluid velocity field $\mathbf{u}:(t, \mathbf{x})\in I \times \Oeta \mapsto  \mathbf{u}(t, \mathbf{x}) \in \mathbb{R}^3$, a pressure function $p:(t, \mathbf{x})\in I \times \Oeta \mapsto  p(t, \mathbf{x}) \in \mathbb{R}$ 
and an extra stress tensor $\bT :(t, \mathbf{x} )\in I \times \Oeta  \mapsto \bT (t, \mathbf{x} ) \in \mathbb{R}^{3\times3}$
 such that the system of equations 
\begin{align}
\divx \bu=0, \label{divfree} 
\\
\partial_t \bu  + (\mathbf{u}\cdot \nabx)\mathbf{u} 
= 
\delx \bu -\nabx p
+\bff
+
\divx   \bT, \label{momEq}
\\
\partial_t^2 \eta - \partial_t\dely \eta + \dely^2 \eta = g - ( \mathbb{S}\bn_\eta )\circ \bm{\varphi}_\eta\cdot\bn \,\vert\det(\naby\bm{\varphi}_\eta)\vert, 
\label{shellEQ}
\\
\partial_t \bT + (\mathbf{u}\cdot \nabx) \bT
=
(\nabx \bu)\bT + \bT(\nabx \bu)^\top+2\mathbb{D}(\nabx \bu) - 2 \bT  \label{solute}
\end{align}
holds in $I\times\Oeta\subset \mathbb R^{1+3}$ (with \eqref{shellEQ} defined on $I\times\omega\subset \mathbb R^{1+2}$).  We complement \eqref{divfree}--\eqref{solute} with the following initial and boundary conditions
\begin{align}
&\eta(0,\cdot)=\eta_0(\cdot), \qquad\partial_t\eta(0,\cdot)=\eta_\star(\cdot) & \text{on }\omega,
\\
&\bu(0,\cdot)=\bu_0(\cdot) & \text{in }\Omega_{\eta_0},
\\
&\bT(0,\cdot)=\bT_0(\cdot) &\text{in }\Omega_{\eta_0},
\label{initialCondSolv}
\\
\label{interface}
&\bu\circ\bm{\varphi}_\eta=(\partial_t\eta)\bn & \text{on }I\times \omega.
\end{align}
Note that the first two terms on the right of \eqref{solute} can be rewritten, in perhaps, the more familiar form
\begin{align*}
(\nabx \bu)\bT + \bT(\nabx \bu)^\top =Q(\bT,\nabx\bu)
\end{align*}
where
\begin{align*}
Q(\bT,\nabx\bu)
=
\bT\bD(\nabx \bu)+\bD(\nabx \bu)\bT +
\bW(\nabx \bu)\bT - \bT\bW(\nabx \bu),
\end{align*}
with
\begin{align*}
&\bW(\nabx \bu)=\tfrac{1}{2}\big(\nabx \bu-(\nabx \bu)^\top \big)
\end{align*}
being the skew-symmetric or anti-symmetric part of the velocity gradient $\nabx\bu$.

As mentioned earlier, the three-states analysis of polymeric fluid-structure interaction problems  is a recent program started in \cite{breit2021incompressible}. While classical fluid-structure interaction problems \eqref{x1}-\eqref{x2}  are concerned with the interplay between a solvent and a structure, polymeric fluid-structure interaction problems involve the mutual interactions between a solute, a solvent, and a structure. Analyzing such a system, therefore, requires fundamental techniques from two previously distinct areas of continuum mechanics: the analysis of polymeric fluids and the study of fluid-structure interaction problems. In the initial work \cite{breit2021incompressible}, the solute subcomponent is modelled on the mesoscopic scale leading to a Fokker--Planck equation for the (mesoscopic) description of the polymer distribution. The more general Koiter shell is also used to model the structure and the existence of a weak solution to the underlying system is constructed. When the same structure equation \eqref{shellEQ} is used instead of the Koiter shell, a unique strong solution to the mixed mesoscopic-macroscopic model is then constructed in \cite{breit2023existence}.

In the 2D-2D-1D setting for the fully macroscopic problem, weak and strong solutions are constructed in \cite{mensah2023weak}. An important feature making this possible is a dissipation term in the solute subcomponent which is not available in our current setting. In fact, the absence of this term is a crucial reason why we are unable to construct a weak solution in this work. Having constructed these two main classes of solutions, we perform certain singular limit arguments for this dissipative system in our other works \cite{mensah2025equilibration, mensah2024vanishing}.  In \cite{mensah2024vanishing},  we showed that
any family of strong solutions that are parametrized by the center-of-mass diffusion coefficient converges, as the coefficient goes to zero, to a weak solution of a corotational polymeric
fluid-structure interaction system without center-of-mass diffusion but whose polymer number density and extra stress tensor are essentially bounded. As a consequence, a weak-strong uniqueness result is obtained that says that the weak solution of the latter is unique in the class of  strong solutions of the former as the center-of-mass diffusion vanishes.

A similar singular limit argument is also performed in \cite{mensah2025equilibration} except that the limit system is now purely non-dissipative. As such, the damping term in the solute is eliminated together with the center-of-mass diffusion. Weak-strong uniqueness is also shown to be true and the rate at which the solution to the original dissipative system decays exponentially in time to the equilibrium solution is also given.

Coming back to our current work for the 3D-3D-2D problem without diffusion in the solute, we aim to construct relatively high regular solutions that satisfy the whole system pointwise almost everywhere. Beyond the difficulty posed by the lack of diffusion, there are other reasons why constructing less regular and more generalised solutions appear to be beyond reach. One such reason is the hyperbolic structure exhibited by the solute. Several systems with this structure, including our own, that are posed  in a $d$-dimensional spatial domain, inadvertently encounter what we call the $W^{s,2}$-problem where spatially speaking, the solution must belong to $W^{s,2}$ for $s>d/2$. This immediately imposes a fairly high regularity threshold on the solute and since our system is highly coupled, this imposition transfers to the solvent and structure subsystems.

Another challenge emanates from the very 3D-3D-2D setting where several optimal interpolation that one obtains in the 2D-2D-1D setting are lost. This is particularly critical for the solute subsystem where the 3D Ladyzhenskaya's inequality applied to the terms $(\nabx \bu)\bT + \bT(\nabx \bu)^\top$ in the convected derivative yield undesirable estimates when working in lower regularity spaces.

Due to these reasons, we construct a higher regularity notion of a strong solution on the level of Definition \ref{def:strongSolutionAlone} for just the solvent-structure interaction problem. The precise definition is:

\begin{definition}[Strong solution]
\label{def:strongSolution}
Let $(\bff, g,\eta_0, \eta_\star,  \bu_0, \bT_0)$
be a dataset that satisfies
\begin{equation}
\begin{aligned}
\label{mainDataForAllStrong}
&\bff \in W^{1,2}(I;L^{2}_\mathrm{loc}(\Oeta))\cap L^{2}(I;W^{2,2}_\mathrm{loc}(\Oeta)),
\qquad 
\bff(0)\in W^{1,2}(\Omega_{\eta_0}),
\\&g\in L^2(I;W^{2,2}(\omega)) \cap  W^{1,2}(I;W^{1,2}(\omega)) ,
\\&
\eta_0 \in W^{5,2}(\omega) \text{ with } \Vert \eta_0 \Vert_{L^\infty( \omega)} < L, \quad \eta_\star \in W^{3,2}(\omega),
\\&\bu_0 \in W^{3,2}_{\divx}(\Omega_{\eta_0} )\text{ is such that }\bu_0 \circ \bm{\varphi}_{\eta_0} =\eta_\star \bn \text{ on } \omega,
\\& 
\bT_0\in W^{3,2}(\Omega_{\eta_0}).
\end{aligned}
\end{equation}
We call 
$(\eta, \bu, p,  \bT)$ 
a \textit{strong solution} of   \eqref{divfree}--\eqref{interface} with dataset $(\bff, g,\eta_0, \eta_\star,  \bu_0, \bT_0)$  if:
\begin{itemize}
\item[(a)] the structure-function $\eta$ is such that $
\Vert \eta \Vert_{L^\infty(I \times \omega)} <L$ and 
\begin{align*}
\eta \in W^{1,\infty}\big(I;W^{3,2}(\omega)  \big)\cap L^{\infty}\big(I;W^{5,2}(\omega)  \big) \cap  W^{3,2}\big(I; L^{2}(\omega)  \big)\cap  W^{2,2}\big(I;W^{2,2}(\omega)  \big) \cap  L^2\big(I;W^{6,2}(\omega)  \big);
\end{align*}
\item[(b)] the velocity $\bu$ is such that $\bu  \circ \bm{\varphi}_{\eta} =(\partial_t\eta)\bn$ on $I\times \omega$ and
\begin{align*} 
\bu\in  W^{2,2} \big(I; L^{2}_{\divx}(\Oeta ) \big)\cap W^{1,2} \big(I; W^{2,2}(\Oeta ) \big)\cap L^2\big(I;W^{4,2}(\Oeta)  \big);
\end{align*}
\item[(c)] the pressure $p$ is such that 
\begin{align*}
p\in W^{1,2} \big(I; W^{1,2}(\Oeta ) \big)\cap L^2\big(I;W^{3,2}(\Oeta)  \big);
\end{align*}
\item[(d)] the tensor $\bT$ is such that 
\begin{align*}
 \bT \in W^{1,\infty}\big(I;W^{2,2}(\Oeta)  \big) \cap  L^\infty\big(I;W^{3,2}(\Oeta)  \big);
\end{align*}
\item[(e)] equations \eqref{divfree}--\eqref{solute} are satisfied almost everywhere in spacetime with $\eta(0)=\eta_0$ and $\partial_t\eta(0)=\eta_\star$ almost everywhere on $\omega$, as well as $\bfu(0)=\bfu_0$ and $\bT(0)=\bT_0$ almost everywhere in $\Omega_{\eta_0}$.
\end{itemize}
\end{definition}

With this definition in hand, we now state our main result.
\begin{theorem}\label{thm:MAIN}
Let $T>0$ be given. 
Let  $(\eta,\bu,p)$ be a unique local strong solution of \eqref{x1}-\eqref{x2}, in the sense of
Definition \ref{def:strongSolutionAlone}, with dataset $(\bff, g, \eta_0, \eta_\star, \bu_0, \underline{\bT})$ satisfying \eqref{datasetAlone} and 
\begin{equation}
\begin{aligned} 
 c&\big( 
 \Vert \eta_\star\Vert_{W^{1,2}(\omega)}^2
 +
  \Vert \eta_0\Vert_{W^{3,2}(\omega)}^2
  +
  \Vert  \bu_0\Vert_{W^{1,2}(\Omega_{\eta_0})}^2
  \big) 
   \\&\quad+ \int_{I } \big(  \Vert \bff\Vert_{L^2(\Omega_\eta)}^2  +
    \Vert \underline{\bT}\Vert_{W^{1,2}(\Omega_\eta)}^2  +
 \Vert g\Vert_{L^2(\omega)}^2 \big)\dt
 \leq
 \epsilon
\end{aligned}
\end{equation}
for some strictly positive constant $\epsilon<1$. 
Then there exist a strictly positive time $T_*\leq T$ such that \eqref{divfree}--\eqref{interface}  admits a unique strong solution $(\eta,\bu,p,\bT)$, in the sense of
Definition \ref{def:strongSolution}, with dataset $(\bff, g, \eta_0, \eta_\star, \bu_0,  \bT_0)$ satisfying \eqref{mainDataForAllStrong}.  
\end{theorem}

\subsection{Plan} The strategy for solving Theorem \ref{thm:MAIN} involves splitting the problem into a solvent-structure subproblem and a solute subproblem. A fixed-point argument for composite maps is then used to  glue the solutions of these subproblems together.
As stated in Theorem \ref{thm:BM}, we have recently constructed the required notion of a strong solution for the first subproblem. To combine this solution with the solution of the second subproblem, however, we require non-trivial higher order estimates beyond the natural energy estimate. This raises an issue of maximal regularity for the Stokes problem on a moving domain with non-trivial boundary conditions; a result that is currently unavailable in the literature. For this reason, we devote an intermediate section, Section \ref{sec:maxReg}, to the proof of this result. The main result is given by Theorem \ref{thm:stokeInhom} below and it extends the recent breakthrough \cite{breit2025partial} for the same moving boundary Stokes problem but with the trivial boundary condition. By combining these maximal regularity results with elliptic estimates for the pressure, we are able to derive our desired estimate for the solvent-structure subproblem under a smallness condition on the initial acceleration.

Due to the lack of dissipation in the solute subproblem, we find it difficult to use a Galerkin method to construct a solution for the solute subproblem. We, therefore, resort to the method of characteristics. To apply this to our equation, however, we apply a regularity preserving transformation to rectify the disorder in the convected derivative term $\bT(\nabx \bu)^\top$. By `disorder', we simply mean that  we ideally wanted the terms in the product swapped. Since matrix products are not commutative, we transform the whole tensor-valued equation for the solute subproblem into a vector-valued equation which allows us to apply the aforementioned method of characteristics while preserving regularity.

Having constructed solutions to the two subproblems, we then apply Banach fixed-point arguments to  get a local strong solution to the fully coupled system.  We mention in passing that the high-regularity required for our solution makes it difficult, if at all possible, to hope to extend the local solution for the full system globally (assuming no degeneracies even occur in the shell) even with just a small initial acceleration.

\section{Preliminaries}
\label{sec:prelims}
For any two non-negative quantities $F$ and $G$, we write $F \lesssim G$  if there is a constant $c>0$  such that $F \leq c\,G$. If $F \lesssim G$ and $G\lesssim F$ both hold, we use the notation $F\sim G$.  The scalar matrix product of two matrices $\mathbb{A}=(a_{ij})_{i,j=1}^d$ and $\mathbb{B}=(b_{ij})_{i,j=1}^d$ is denoted by $\mathbb{A}:\mathbb{B}=\sum_{ij}a_{ij}b_{ij}$.
The symbol $\vert \cdot \vert$ may be used in four different contexts. For a scalar function $f\in \mathbb{R}$, $\vert f\vert$ denotes the absolute value of $f$. For a vector $\bff\in \mathbb{R}^d$, $\vert \bff \vert$ denotes the Euclidean norm of $\bff$. For a square matrix $\mathbb{F}\in \mathbb{R}^{d\times d}$, $\vert \mathbb{F} \vert$ shall denote the Frobenius norm $\sqrt{\mathrm{trace}(\mathbb{F}^T\mathbb{F})}$. Also, if $S\subseteq  \mathbb{R}^d$ is  a (sub)set, then $\vert S \vert$ is the $d$-dimensional Lebesgue measure of $S$.
Since we only consider functions on $\omega \subset\mathbb{R}^2$ with periodic boundary
conditions and zero mean values, we have the following equivalences
\begin{equation}
\begin{aligned}\nonumber
\Vert \cdot\Vert_{W^{1,2}(\omega)}\sim
\Vert \naby\cdot\Vert_{L^{2}(\omega)},
\qquad
\Vert \cdot\Vert_{W^{2,2}(\omega)}\sim
\Vert \Dely\cdot\Vert_{L^{2}(\omega)},
\\
\Vert \cdot\Vert_{W^{3,2}(\omega)}\sim
\Vert \naby\Dely\cdot\Vert_{L^{2}(\omega)},
\qquad
\Vert \cdot\Vert_{W^{4,2}(\omega)}\sim
\Vert \Dely^2\cdot\Vert_{L^{2}(\omega)},
\\
\Vert \cdot\Vert_{W^{5,2}(\omega)}\sim
\Vert \naby\Dely^2\cdot\Vert_{L^{2}(\omega)},
\qquad
\Vert \cdot\Vert_{W^{6,2}(\omega)}\sim
\Vert \Dely^3\cdot\Vert_{L^{2}(\omega)}.
\end{aligned} 
\end{equation}
For $I:=(0,T)$, $T>0$, and $\eta\in C(\overline{I}\times\omega)$ satisfying $\|\eta\|_{L^\infty(I\times\omega)}\leq L$ where $L>0$ is a constant, we define for $1\leq p,r\leq\infty$,
\begin{align*} 
L^p(I;L^r(\Omega_\eta))&:=\Big\{v\in L^1(I\times\Omega_\eta):\substack{v(t,\cdot)\in L^r(\Omega_{\eta(t)})\,\,\text{for a.e. }t,\\\|v(t,\cdot)\|_{L^r(\Omega_{\eta(t)})}\in L^p(I)}\Big\},\\
L^p(I;W^{1,r}(\Omega_\eta))&:=\big\{v\in L^p(I;L^r(\Omega_\eta)):\,\,\nabx v\in L^p(I;L^r(\Omega_\eta))\big\}.
\end{align*} 
Higher-order Sobolev spaces can be defined accordingly. For $k>0$ with $k\notin\mathbb N$, we define the fractional Sobolev space $L^p(I;W^{k,r}(\Oeta))$ as the class of $L^p(I;L^r(\Omega_\eta))$-functions $v$ for which 
\begin{align*}
\|v\|_{L^p(I;W^{k,r}(\Oeta))}^p
&=\int_I\bigg(\int_{\Oeta} \vert v\vert^r\dx
+\int_{\Oeta}\int_{\Oeta}\frac{|v(\bx)-v(\bx')|^r}{|\bx-\bx'|^{d+k r}}\dx\dx'\bigg)^{\frac{p}{r}}\dt
\end{align*}
is finite. Accordingly, we can also introduce fractional differentiability in time for the spaces on moving domains.

\subsection{Setup}
\label{sec:setup}
The reference spatial domain  $\Omega \subset \mathbb{R}^3$ has a boundary $\partial\Omega\subset \mathbb{R}^{2}$ that may consist of a flexible part $\omega\subset\mathbb{R}^{2}$ and a rigid part $\Gamma\subset \mathbb{R}^{2}$. However, because the analysis at the rigid part is essentially trivial when compared to the fixed part, we shall identify the whole of $\partial \Omega$ with $\omega$ and endow it with periodic boundary conditions. Let $I:=(0,T)$ represent a time interval for a given constant $T>0$. The time-dependent  displacement of the structure is given by $\eta:\overline{I}\times\omega\rightarrow(-L,L)$ where $L>0$ is a fixed length of the tubular neighbourhood of $\partial\Omega$ given by
\begin{align*}
S_L:=\{\bx\in \mathbb{R}^3\,:\, \mathrm{dist}(\bx,\partial\Omega
)<L \}.
\end{align*}
For some $k\in\mathbb{N}$ large enough, we now assume that $\partial\Omega$  is parametrized by an injective mapping $\bm{\varphi}\in C^k(\omega;\mathbb{R}^3)$ with $\naby \bm{\varphi}\neq0$ and small Lipschitz bound, i.e. there exists a strictly positive constant $\epsilon<1$ such that $\Vert \bm{\varphi}\Vert_{W^{1,\infty}(\omega)}\leq \epsilon$.
The set $\partial{\Omega_{\eta(t)}}$, given by 
\begin{align*}
\partial{\Omega_{\eta(t)}}=\big\{\bm{\varphi}_{\eta(t)}:=\bm{\varphi}(\by)+\bn(\by)\eta(t,\by)\, :\, t\in I, \by\in \omega\big\},
\end{align*}
then represents the boundary of the flexible (moving) domain at any instant of time $t\in I$ and the vector $\bn(\by)$ is the outward unit normal at the point $\by\in \omega$. 
We also let $\mathbf{n}_{\eta(t)}(\by) $
be the corresponding outward  normal of $\partial{\Omega_{\eta(t)}}$ at the spacetime point $\by\in \omega$ and $t\in I$. Then for $L>0$ sufficiently small, we note that $\bn_{\eta(t)}(\by)$ is close to $\bn(\by)$ and $\bm{\varphi}_{\eta(t)}$ is close to $\bm{\varphi}$. As a result,  it follows that
\begin{align}
\label{eq:1705}
\partial_{y_1} \bm{\varphi}_{\eta(t)}\times
\partial_{y_2} \bm{\varphi}_{\eta(t)} \neq\bm{0} \quad\text{ and } \quad
\bn(\by)\cdot \bn_{\eta(t)}(\by)\neq 0
\end{align}
for $\by\in \omega$ and $t\in I$. Thus, in particular, there is no loss of strict positivity of the Jacobian determinant provided that $\Vert \eta\Vert_{L^\infty(I;W^{1,\infty}(\omega))}<L$.

For the interior points, we  transform the  reference domain $\Omega$ into a time-dependent moving domain $\Omega_{\eta(t)}$  whose state at time $t\in\overline{I}$ is given by
\begin{align*}
\Omega_{\eta(t)}
 =\big\{
 \bm{\Psi}_{\eta(t)}(\bx):\, \bx \in \Omega 
  \big\}.
\end{align*}
Here,
\begin{align*}
\bm{\Psi}_{\eta(t)}(\bx)=
\begin{cases}
\bx+\bn(\by(\bx))\eta(t,\by(\bx))\phi(s(\bx))     & \quad \text{if } \mathrm{dist}(\bx,\partial\Omega)<L,\\
    \bx & \quad \text{elsewhere } 
  \end{cases}
\end{align*}
is the Hanzawa transform with inverse $\bm{\Psi}_{-\eta(t)}$ and where for a point $\bx$ in the neighbourhood of $\partial\Omega$, the vector $\bn(\by(\bx))$ is the outward unit normal at the point $\by(\bx)=\mathrm{arg min}_{\by\in\omega}\vert\bx -\bm{\varphi}(\by)\vert$. Also, $s(\bx)=(\bx-\bm{\varphi}(\by(\bx)))\cdot\bn(\by(\bx))$ and $\phi\in C^\infty(\mathbb{R})$ is a cut-off function that is $\phi\equiv0$ in the neighbourhood of $-L$ and $\phi\equiv1$ in the neighbourhood of $0$. Note that $\bm{\Psi}_{\eta(t)}(\bx)$ can be rewritten as
\begin{align*}
\bm{\Psi}_{\eta(t)}(\bx)=
\begin{cases}
\bm{\varphi}(\by(\bx))+\bn(\by(\bx))[s(\bx)+\eta(t,\by(\bx))\phi(s(\bx)) ]    & \quad \text{if } \mathrm{dist}(\bx,\partial\Omega)<L,\\
    \bx & \quad \text{elsewhere. } 
  \end{cases}
\end{align*}
We also note that for the boundary point $\bm{\varphi}(\by)\in \partial\Omega\subset\R^3$, we have that
\begin{align*}
\big(\bm{\varphi}_\eta\circ\bm{\varphi}^{-1}\big)(\bm{\varphi}(\by))=\bm{\varphi}_\eta(\by)=\bm{\varphi}(\by)+\bn\eta\in \partial\Oeta.
\end{align*}
Therefore, $\bm{\varphi}_\eta\circ\bm{\varphi}^{-1}:\partial\Omega\rightarrow\partial\Oeta$ corresponds to the Hanzawa transform $\bm{\Psi}_\eta:\overline{\Omega}\rightarrow\overline{\Oeta}$ restricted to the boundary, i.e., $\bm{\Psi}_\eta\vert_{\partial\Omega}=\bm{\varphi}_\eta\circ\bm{\varphi}^{-1}$. Consequently, in particular,
\begin{align}
\label{bdd3trans}
\bm{\varphi}_\eta=\bm{\Psi}_\eta\circ\bm{\varphi}&\quad \text{on } I \times \omega.
\end{align}
 Furthermore, the transform $\bm{\Psi}_\eta$ and its inverse 
 $\bm{\Psi}_\eta=\bm{\Psi}_{-\eta}$ satisfies the following properties, see \cite{breit2022regularity, BMSS} for details.
 If for some $\ell,R>0$, we assume that
\begin{align*}
\Vert\eta\Vert_{L^\infty(\omega)}
+
\Vert\zeta\Vert_{L^\infty(\omega)}
< \ell <L \qquad\text{and}\qquad
\Vert\naby\eta\Vert_{L^\infty(\omega)}
+
\Vert\naby\zeta\Vert_{L^\infty(\omega)}
<R
\end{align*}
holds, then for any  $s>0$, $\varrho,p\in[1,\infty]$ and for any $\eta,\zeta \in B^{s}_{\varrho,p}(\omega)\cap W^{1,\infty}(\omega)$ (where $B^{s}_{\varrho,p}$ is a Besov space), we have that the estimates
\begin{align}
\label{210and212}
&\Vert \bm{\Psi}_\eta \Vert_{B^s_{\varrho,p}(\Omega\cup S_\ell)}
+
\Vert \bm{\Psi}_\eta^{-1} \Vert_{B^s_{\varrho,p}(\Omega\cup S_\ell)}
 \lesssim
1+ \Vert \eta \Vert_{B^s_{\varrho,p}(\omega)},
\\
\label{211and213}
&\Vert \bm{\Psi}_\eta - \bm{\Psi}_\zeta  \Vert_{B^s_{\varrho,p}(\Omega\cup S_\ell)} 
+
\Vert \bm{\Psi}_\eta^{-1} - \bm{\Psi}_\zeta^{-1}  \Vert_{B^s_{\varrho,p}(\Omega\cup S_\ell)} 
\lesssim
 \Vert \eta - \zeta \Vert_{B^s_{\varrho,p}(\omega)}
\end{align}
and
\begin{align}
\label{218}
&\Vert \partial_t\bm{\Psi}_\eta \Vert_{B^s_{\varrho,p}(\Omega\cup S_\ell)}
\lesssim
 \Vert \partial_t\eta \Vert_{B^{s}_{ \varrho,p}(\omega)},
\qquad
\eta
\in W^{1,1}(I;B^{s}_{\varrho,p}(\omega))
\end{align}
holds uniformly in time with the hidden constants depending only on the reference geometry, on $L-\ell$ and $R$.

\section{Maximal regularity theory}
\label{sec:maxReg}
A central result we shall use in our construction of a solution, which is of independent interest, is the extension of the classical maximal regularity theory for the Stoke system to moving domains with nontrivial boundary conditions. This is a new result and builds on the recent work \cite[Theorem 3.1]{breit2025partial} for moving domains with homogeneous boundary conditions. More precisely, we consider the following system
\begin{equation}
\begin{aligned}
\label{stokeInhom}
&\partial_t \bu   
= 
\delx \bu -\nabx p
+\bff 
&\text{ in } I\times\Oeta,
\\
&\divx\bu=0
&\text{ in } I\times\Oeta,
\\&
\bu\circ\bm{\varphi}_\eta =(\partial_t\eta)\bn
&\text{ on } I\times \omega,
\\&
\bu(0)=\bu_0
&\text{ in }  \Omega_{\eta_0}
\end{aligned}
\end{equation}
and prove the following result.
\begin{theorem}\label{thm:stokeInhom}
Let $(\bff, \eta_0, \eta_\star,  \bu_0, \eta)$
be a dataset such that
\begin{align*}
&\bff \in L^2(I\times\Oeta), \qquad
\eta(0,\cdot)=\eta_0(\cdot)\in W^{3,2}(\omega),
\\&
\partial_t\eta (0,\cdot)=\eta_\star(\cdot)\in W^{1,2}(\omega), \qquad
\bu(0,\cdot)=\bu_0(\cdot)\in W^{1,2}_{\divx}(\Omega_{\eta_0}),
\\&
\text{with }\bu_0 \circ \bm{\varphi}_{\eta_0} =\eta_\star\bn \text{   on } \omega.
\end{align*} 
Suppose further that
\begin{align*}
\eta\in W^{2,2}(I;L^{2}(\omega))
\cap
W^{1,\infty}(I;W^{1,2}(\omega))
\cap 
W^{1,2}(I;W^{2,2}(\omega))
\cap
L^\infty(I;W^{3,2}(\omega))
\end{align*} 
and  there exists constants $\ell,\kappa_0>0$ such that
\begin{align*}
\Vert\eta\Vert_{L^\infty(I\times\omega)}\leq \ell<L\qquad
\inf_{I\times\omega}\vert\partial_1\bm{\varphi}_\eta\times \partial_2 \bm{\varphi}_\eta\vert \geq \kappa_0.
\end{align*} 
Then there exist a unique solution to \eqref{thm:stokeInhom} such that the estimate
\begin{align*}
\int_I\big(\Vert \partial_t\bu\Vert_{L^2(\Oeta)}^2+&\Vert  \bu\Vert_{W^{2,2}(\Oeta)}^2
+
\Vert p\Vert_{W^{1,2}(\Oeta)}^2\big)\dt
\\
\lesssim&
\Vert\bu_0\Vert_{W^{1,2}(\Omega_{\eta_0})}^2
+
\Vert\eta_\star\Vert_{W^{1,2}(\omega)}^2
+
\int_I\big(\Vert\bff \Vert_{L^2(\Oeta)}^2
+
 \|\partial_t^2\eta \|_{L^{2}( \omega )}^2 
\big)\dt
\\&+
\int_I
\big( 
 \| |\partial_t \eta|^2 \|_{L^{2}( \omega )}^2
+
\|\partial_t\eta \|_{W^{2,2}( \omega )}^2
+
\|\partial_t\eta\,\naby^2 \eta \|_{L^{2}( \omega )}^2
\big)\dt
\\&\qquad+
\int_I
\big(
 \|\naby\partial_t \eta\,\naby \eta\|_{L^{2}( \omega )}^2 
+
\|\partial_t\eta\,|\naby \eta|^2 \|_{L^{2}( \omega )}^2
\big)\dt
\end{align*}
hold with a constant depending on $\eta_0$.
\end{theorem}
Theorem \ref{thm:stokeInhom} extend the following 
maximal regularity result \cite[Theorem 3.1]{breit2025partial} for the Stokes system in a moving domain but with zero boundary condition.
\begin{theorem}\label{thm:stokeHom}
Let $(\bff, \eta_0,    \bu_0, \eta)$
be a dataset such that
\begin{align*}
&\bff \in L^2(I\times\Oeta), \qquad
\eta(0,\cdot)=\eta_0(\cdot)\in W^{3,2}(\omega),
\\& 
\bu(0,\cdot)=\bu_0(\cdot)\in W^{1,2}_{0,\divx}(\Omega_{\eta_0}),
\\&
\text{with }\bu_0 \circ \bm{\varphi}_{\eta_0} =\bm{0} \text{   on } \omega.
\end{align*} 
Suppose further that
\begin{align*}
\eta\in C(\overline{I}; B_{\varrho,2}^\theta(\omega))
\cap
 C(\overline{I};W^{2,2}(\omega))
\cap 
 C(\overline{I};C^{1}(\omega))
\cap
W^{1,3}(I;W^{3,2}(\omega))
\end{align*} 
with $\varrho\geq4$ and $\theta\in(\tfrac{3}{2}, \tfrac{3}{2}+\tfrac{3}{\varrho}]$, and  there exists constants $\ell,\kappa_0>0$ such that
\begin{align*}
\Vert\eta\Vert_{L^\infty(I\times\omega)}\leq \ell<L\qquad
\inf_{I\times\omega}\vert\partial_1\bm{\varphi}_\eta\times \partial_2 \bm{\varphi}_\eta\vert \geq \kappa_0.
\end{align*} 
%
Then there exist a unique solution to
\begin{align*}
&\partial_t \bu   
= 
\delx \bu -\nabx p
+\bff 
&\text{ in } I\times\Oeta,
\\
&\divx\bu=0
&\text{ in } I\times\Oeta,
\\&
\bu\circ\bm{\varphi}_\eta =\bm{0}
&\text{ on } I\times\omega,
\\&
\bu(0)=\bu_0
&\text{ in }  \Omega_{\eta_0}
\end{align*}
such that
\begin{align*}
\int_I\big(\Vert \partial_t\bu\Vert_{L^2(\Oeta)}^2+\Vert  \bu\Vert_{W^{2,2}(\Oeta)}^2
+
\Vert p\Vert_{W^{1,2}(\Oeta)}^2\big)\dt
\lesssim
\Vert\bu_0\Vert_{W^{1,2}(\Omega_{\eta_0})}^2
+
\int_I\Vert\bff\Vert_{L^2(\Oeta)}^2.
\end{align*}
\end{theorem}
\begin{remark}
Note that the restrictive choice of $W^{2,p}\cap W^{1,p}_{\divx} $ for the initial condition $\bu_0$ in \cite[
Theorem 3.1]{breit2022regularity} is just a simplification to avoid the usual but complicated choice of $\bu_0\in (L^p,W^{2,p})_{1-\frac{1}{r},r}\cap L^p_{\divx}$ that is required for the theorem to hold true. Here, $(L^p,W^{2,p})_{1-\frac{1}{r},r}=B^{2(1-1/r)}_{p,r}$ is an interpolation space. In the particular case where $r=p=2$  in \cite[Theorem 3.1]{breit2022regularity}, a careful inspection of the proof shows that $\bu_0\in  W^{1,2}_{0,\divx}(\Omega_{\eta_0})$ is sufficient for the proof as expected since  $B^{2(1-1/r)}_{p,r}=W^{1,2}$ when $r=p=2$.
\end{remark}
The key to the proof of Theorem \ref{thm:stokeInhom} is the  combination of Theorem \ref{thm:stokeHom} with the following proposition (see \cite[Proposition 3.3]{muha2019existence}) concerning solenoidal extensions on moving domains and their properties.
\begin{proposition}
\label{thm:extension}
For a given $\eta\in L^\infty(I;W^{1,2}( \omega ))$ with $\|\eta\|_{L^\infty(I\times\omega)}\leq\ell<L$, there are linear operators
\begin{align*}
\mathscr K_\eta:L^1( \omega )\rightarrow\mathbb R,\quad
\Testzeta:\{\xi\in L^1(I;W^{1,1}( \omega )):\,\mathscr K_\eta(\xi)=0\}\rightarrow L^1(I;W^{1,1}_{\Div}(\Omega\cup S_{\ell} )),
\end{align*}
such that the tuple $(\Testzeta(\xi-\mathscr K_\eta(\xi)),\xi-\mathscr K_\eta(\xi))$ satisfies
\begin{align*}
\Testzeta(\xi-\mathscr K_\eta(\xi))&\in L^\infty(I;L^2(\Omega_\eta))\cap L^2(I;W^{1,2}_{\Div}(\Omega_\eta)),
\\
\xi-\mathscr K_\eta(\xi)&\in L^\infty(I;W^{2,2}( \omega ))\cap  W^{1,\infty}(I;L^{2}( \omega )),
\\
\partial_t(\Testzeta&(\xi-\mathscr K_\eta(\xi))) \in L^2(I;L^2(\Omega_\eta)),
\\
  (\Testzeta(\xi- \mathscr K_\eta&(\xi)))\circ \bm{\varphi}_\eta =(\xi-\mathscr K_\eta(\xi))\bn,
\quad\text{ on } I\times\omega,
\\
\Testzeta(\xi-\mathscr K_\eta&(\xi))(t,\bx)=0 \text{ for } (t,\bx)\in I \times (\Omega \setminus S_{\ell})
\end{align*}
provided that we have $\xi\in L^\infty(I;W^{2,2}( \omega ))\cap  W^{1,\infty}(I;L^{2}(\omega))$.
In particular, we have the estimates 
\begin{align*}
\|\Testzeta(\xi-\mathscr K_\eta(\xi))\|_{L^q(I;W^{1,p}(\Omega \cup S_{\ell}  ))}
&\lesssim \|\xi\|_{L^q(I;W^{1,p}( \omega ))}+\|\xi\naby \eta\|_{L^q(I;L^{p}( \omega ))},
\\ 
\|\partial_t\Testzeta(\xi-\mathscr K_\eta(\xi))\|_{L^q(I;L^{p}( \Omega\cup S_{\ell}))}
&\lesssim \|\partial_t\xi\|_{L^q(I;L^{p}( \omega ))}+\|\xi\partial_t \eta\|_{L^q(I;L^{p}( \omega ))},
\\
\|\Testzeta(\xi-\mathscr K_\eta(\xi))\|_{L^q(I;W^{2,p}(\Omega \cup S_{\ell}  ))}
&\lesssim \|\xi\|_{L^q(I;W^{2,p}( \omega ))}+\|\xi\naby^2 \eta\|_{L^q(I;L^{p}( \omega ))}
\\
&\quad+\|  \naby\xi\,\naby \eta\|_{L^q(I;L^{p}( \omega ))}+\|\xi \,|\naby \eta|^2\|_{L^q(I;L^{p}( \omega ))},
\\ 
\|\partial_t\Testzeta(\xi-\mathscr K_\eta(\xi))\|_{L^q(I;W^{1,p}( \Omega\cup S_{\ell}))}
&\lesssim \|\partial_t\xi\|_{L^q(I;W^{1,p}( \omega ))}+\|\xi\partial_t \naby \eta\|_{L^q(I;L^{p}( \omega ))}
\\
&\quad+\|  \partial_t \xi\,\naby \eta \|_{L^q(I;L^{p}( \omega ))}+\|\naby \xi\,\partial_t \eta \|_{L^q(I;L^{p}( \omega ))}
\\
&\quad+\|\xi\,\partial_t\eta\,\naby \eta|\|_{L^q(I;L^{p}( \omega ))},
\end{align*}
for any $p\in (1,\infty),q\in[1,\infty]$.
\end{proposition}
\begin{remark}
The operator $\mathscr K_\eta$ is a \textit{correction} operator that is applied to functions defined on the boundary that do not allow for a solenoidal extension.
If $\xi$ is the trace of a solenoidal vector field $\bu$ in $ \Oeta$, however, then $\mathscr K_\eta(\xi)=0$ (see \cite[Lemma 6.3]{muha2019existence} for more details) 
and the difference $ \bu-\Testzeta(\xi)$
also has zero trace on $\partial\Omega_\eta$. 
\end{remark}

%

\begin{proof}[Proof of Theorem \ref{thm:stokeInhom}]
By Proposition \ref{thm:extension}, there exists an extension $$\Phi:= \Testzeta(\partial_t\eta\bn )$$ such that
\begin{align*}
&\divx\Phi =0 &\text{in } I\times \Oeta,
\\
&\Phi\circ\bm{\varphi}_\eta =\partial_t\eta\,\bn &\text{on }\overline{ I}\times\omega
\end{align*}
and
\begin{equation}
\begin{aligned}
\label{phibdd} 
\int_I\big(&
\|\partial_t\Phi\|_{ L^{2}( \Oeta)}^2
+
\|\Phi\|_{W^{2,2}(\Oeta )}^2
\big)\dt
\\&\lesssim
\int_I
\big(
 \|\partial_t^2\eta\bn\|_{L^{2}( \omega )}^2
 +
 \| |\partial_t \eta|^2\bn\|_{L^{2}( \omega )}^2
+
\|\partial_t\eta\bn\|_{W^{2,2}( \omega )}^2
+
\|\partial_t\eta\,\naby^2 \eta\,\bn\|_{L^{2}( \omega )}^2
\big)\dt
\\&\qquad+
\int_I
\big(
 \|\naby(\partial_t \eta\bn)\naby \eta\|_{L^{2}( \omega )}^2 
+
\|\partial_t\eta\,|\naby \eta|^2\bn\|_{L^{2}( \omega )}^2
\big)\dt.
\end{aligned}
\end{equation} 
Recall from Section \ref{sec:setup} that $\bn$ is a smooth unit vector.

Now set $\bw:=\bu-\Phi$ so that \eqref{stokeInhom} is equivalent to
\begin{equation}
\begin{aligned}
\label{stokeInhomX}
&\partial_t \bw   
= 
\delx \bw -\nabx p
+\bff 
-\partial_t\Phi
+\delx \Phi
&\text{ in } I\times\Oeta,
\\
&\divx\bw=0
&\text{ in } I\times\Oeta,
\\&
\bw\circ\bm{\varphi}_\eta=\bm{0}
&\text{ on } I\times\omega,
\\&
\bw(0)=\bu_0-\Phi(0)
&\text{ in }  \Omega_{\eta_0}.
\end{aligned}
\end{equation}
If we now apply Theorem \ref{thm:stokeHom} to \eqref{stokeInhomX}, we obtain a unique solution satisfying
\begin{align*}
\int_I\big(\Vert \partial_t\bw\Vert_{L^2(\Oeta)}^2+\Vert  \bw\Vert_{W^{2,2}(\Oeta)}^2
+&
\Vert p\Vert_{W^{1,2}(\Oeta)}^2\big)\dt
\\
\lesssim&
\Vert\bu_0\Vert_{W^{1,2}(\Omega_{\eta_0})}^2
+
\Vert\eta_\star\Vert_{W^{1,2}(\omega )}^2
\\&+
\int_I\big(\Vert\bff \Vert_{L^2(\Oeta)}^2
+
\Vert\partial_t\Phi\Vert_{L^2(\Oeta)}^2
+
\Vert \Phi\Vert_{W^{2,2}(\Oeta)}^2\Big)\dt
\end{align*} 
and by using the inequality
\begin{align*}
|x-y|^2\geq (|x|-|y|)^2=|x|^2+|y|^2-2|x|\,|y|
,
\end{align*}
we obtain
\begin{align*}
\int_I\big(\Vert \partial_t\bw\Vert_{L^2(\Oeta)}^2+\Vert  \bw\Vert_{W^{2,2}(\Oeta)}^2 \big)\dt
\geq &
\int_I\big(\Vert \partial_t\bu\Vert_{L^2(\Oeta)}^2+\Vert  \bu\Vert_{W^{2,2}(\Oeta)}^2 \big)\dt
\\&
+
\int_I\big(\Vert \partial_t\Phi\Vert_{L^2(\Oeta)}^2+\Vert  \Phi\Vert_{W^{2,2}(\Oeta)}^2 \big)\dt
\\&
-2
\int_I \Vert \partial_t\Phi\Vert_{L^2(\Oeta)}\Vert \partial_t\bu\Vert_{L^2(\Oeta)} \dt
\\&
-2
\int_I \Vert  \Phi\Vert_{W^{2,2}(\Oeta)}\Vert  \bu\Vert_{W^{2,2}(\Oeta)}
\dt.
\end{align*}
Consequently, it follows that
\begin{align*} 
\int_I&\big(\Vert \partial_t\Phi\Vert_{L^2(\Oeta)}^2+\Vert  \Phi\Vert_{W^{2,2}(\Oeta)}^2 \big)\dt
+
\int_I\big(\Vert \partial_t\bu\Vert_{L^2(\Oeta)}^2+\Vert  \bu\Vert_{W^{2,2}(\Oeta)}^2 \big)\dt
\\
\lesssim&
\Vert\bu_0\Vert_{W^{1,2}(\Omega_{\eta_0})}^2
+
\Vert\eta_\star\Vert_{W^{1,2}(\omega )}^2
+
\int_I\big(\Vert\bff \Vert_{L^2(\Oeta)}^2
+
\Vert\partial_t\Phi\Vert_{L^2(\Oeta)}^2
+
\Vert \Phi\Vert_{W^{2,2}(\Oeta)}^2\Big)\dt
\\&+
\int_I\big(\Vert \partial_t\Phi\Vert_{L^2(\Oeta)}\Vert \partial_t\bu\Vert_{L^2(\Oeta)}+\Vert  \Phi\Vert_{W^{2,2}(\Oeta)}\Vert  \bu\Vert_{W^{2,2}(\Oeta)} \big)\dt.
\end{align*}
We can now drop the $\Phi$ terms on the left-hand side and apply Young's inequality to the last term on the right-hand side as follows:
\begin{align*}
\int_I\big(\Vert \partial_t\Phi\Vert_{L^2(\Oeta)}\Vert \partial_t\bu\Vert_{L^2(\Oeta)}+&\Vert  \Phi\Vert_{W^{2,2}(\Oeta)}\Vert  \bu\Vert_{W^{2,2}(\Oeta)} \big)\dt
\\
\leq
\delta&
\int_I\big(\Vert \partial_t\bu\Vert_{L^2(\Oeta)}^2+\Vert  \bu\Vert_{W^{2,2}(\Oeta)}^2 \big)\dt 
\\&+
c(\delta)
\int_I\big(\Vert \partial_t\Phi\Vert_{L^2(\Oeta)}^2+\Vert  \Phi\Vert_{W^{2,2}(\Oeta)}^2 \big)\dt
\end{align*}
for any $\delta>0$. We can now choose $\delta$ small and absorb this $\delta$-term back into the left-hand side to obtain
\begin{align*}  
\int_I\big(\Vert \partial_t\bu\Vert_{L^2(\Oeta)}^2+\Vert  \bu\Vert_{W^{2,2}(\Oeta)}^2 \big)\dt 
\lesssim&
\Vert\bu_0\Vert_{W^{2,2}(\Omega_{\eta_0})}^2
+
\Vert\eta_\star\Vert_{W^{1,2}(\omega )}^2
\\&+
\int_I\big(\Vert\bff \Vert_{L^2(\Oeta)}^2
+
\Vert\partial_t\Phi\Vert_{L^2(\Oeta)}^2
+
\Vert \Phi\Vert_{W^{2,2}(\Oeta)}^2\Big)\dt.
\end{align*}
However, by \eqref{phibdd},
\begin{align*}
\int_I\big(
\Vert\partial_t\Phi\Vert_{L^2(\Oeta)}^2
+&
\Vert \Phi\Vert_{W^{2,2}(\Oeta)}^2\Big)\dt
\\&\lesssim
\int_I
\big(
 \|\partial_t^2\eta\|_{L^{2}( \omega )}^2
 +
 \| |\partial_t \eta|^2\|_{L^{2}( \omega )}^2
+
\|\partial_t\eta\|_{W^{2,2}( \omega )}^2
+
\|\partial_t\eta\,\naby^2 \eta\|_{L^{2}( \omega )}^2
\big)\dt
\\&\qquad+
\int_I
\big(
 \|\naby\partial_t \eta\naby \eta\|_{L^{2}( \omega )}^2 
+
\|\partial_t\eta\,|\naby \eta|^2\|_{L^{2}( \omega )}^2
\big)\dt
\end{align*}
hold. This finishes the proof.
\end{proof}

In analogy with the standard maximal regularity result for the Stokes system on fixed domains, one can also obtain higher spatial regularity for our moving domain. For the purpose of our work, we are only interested in improving the spatial regularity by two factor of differentiability. For this, we now set $\bm{\Phi}_\eta:=\bm{\Psi}_\eta\circ\bm{\varphi}$ so that the restriction of $\bm{\Phi}_\eta$ to the boundary $\omega$ is $\bm{\varphi}_\eta$, recall \eqref{bdd3trans}. If we take the partial derivatives $\partial_{x_i}\partial_{x_j}$ of $\bu\circ\bm{\Phi}_\eta$, then by the chain rule
\begin{align*}
\partial_{x_i}\partial_{x_j}(\bu\circ\bm{\Phi}_\eta)
&=
(\partial_{x_i}\bm{\Phi}_\eta)^\top(D^2\bu)\circ\bm{\Phi}_\eta\partial_{x_j}\bm{\Phi}_\eta
+
D \bu \circ\bm{\Phi}_\eta\partial_{x_i}\partial_{x_j}\bm{\Phi}_\eta
\end{align*}
with corresponding trace
\begin{align*}
\partial_{y_i}\partial_{y_j}(\bu\circ\bm{\varphi}_\eta)
&=
(\partial_{y_i}\bm{\varphi}_\eta)^\top(D^2\bu)\circ\bm{\varphi}_\eta\partial_{y_j}\bm{\varphi}_\eta
+
D \bu \circ\bm{\varphi}_\eta\naby\bm{\varphi}_\eta (\naby\bm{\varphi}_\eta )^{-1}\partial_{y_i}\partial_{y_j}\bm{\varphi}_\eta
 \\
 &=
\partial_{y_i}\partial_{y_j}\bn\,\partial_t\eta
+
\partial_{y_j}\bn\,\partial_{y_i}\partial_t\eta
+
\partial_{y_j}\bn\,\partial_{y_i}\partial_t\eta
+
\bn\,\partial_{y_i}\partial_{y_j}\partial_t\eta.
\end{align*}
This trace follows from the fact that $ \bm{\Phi}_\eta=\bm{\varphi}_\eta$ and $\bu\circ\bm{\varphi}_\eta =(\partial_t\eta)\bn$  on $I\times \omega$. However, since
\begin{align*}
\naby(\bu\circ\bm{\varphi}_\eta) =D\bu\circ\bm{\varphi}_\eta\naby\bm{\varphi}_\eta =(\naby\partial_t\eta)\otimes\bn+(\partial_t\eta)\naby\bn
\end{align*}
on $I\times \omega$, we can conclude that
\begin{equation}
\begin{aligned}
\label{H}
(D^2\bu)\circ\bm{\varphi}_\eta
=\frac{\partial_{x_i}\bm{\varphi}_\eta \otimes \partial_{x_j}\bm{\varphi}_\eta
}{\vert \partial_{x_i}\bm{\varphi}_\eta\vert^2\vert \partial_{x_j}\bm{\varphi}_\eta\vert^2}
\otimes\big\{&
 \partial_{y_i}\partial_{y_j}\bn\,\partial_t\eta
+
\partial_{y_j}\bn\,\partial_{y_i}\partial_t\eta
+
\partial_{y_j}\bn\,\partial_{y_i}\partial_t\eta
+
\bn\,\partial_{y_i}\partial_{y_j}\partial_t\eta
\\&-
((\naby\partial_t\eta)\otimes\bn+(\partial_t\eta)\naby\bn)
(\naby\bm{\varphi}_\eta )^{-1}\partial_{y_i}\partial_{y_j}\bm{\varphi}_\eta
\big\} .
\end{aligned}
\end{equation}
Since the Stokes equation is linear, it follows that the doubly differentiated Stokes equation for 
\begin{align*}
\bv:=\partial_{x_i}\partial_{x_j}\bu, \qquad
q:=\partial_{x_i}\partial_{x_j} p, \qquad
\mathbf{g}:=\partial_{x_i}\partial_{x_j}\bff
\end{align*}
solve
\begin{equation}
\begin{aligned}
\label{stokeInhomDoubly}
&\partial_t \bv   
= 
\delx \bv -\nabx q
+\mathbf{g}
&\text{ in } I\times\Oeta,
\\
&\divx\bv=0
&\text{ in } I\times\Oeta,
\\&
(D^2\bu)\circ\bm{\varphi}_\eta = \mathbb{H}
&\text{ on } I\times \omega,
\\&
\bv(0)=\bv_0
&\text{ in }  \Omega_{\eta_0},
\end{aligned}
\end{equation}
where $\mathbb{H}$ is the right-hand side of \eqref{H}. We now recall that since $\bm{\varphi}$ and $\bn$ are smooth vectors and $\bm{\varphi}_{\eta}=\bm{\varphi}+\bn\,\eta $, all functions of $\bm{\varphi}_\eta$ and $\bn$ in the definition of $\mathbb{H}$ above are uniformly bounded in $I\times\omega$ if $\eta\in L^\infty(I;W^{3,2}(\omega)$.

If we now apply Theorem \ref{thm:stokeInhom} to \eqref{stokeInhomDoubly} and sum over all $i,j=1,2,3$, an immediate corollary  is the following. 
\begin{corollary}\label{cor:stokeInhom}
Let $(\bff, \eta_0, \eta_\star,  \bu_0, \eta)$
be a dataset such that
\begin{align*}
&\bff \in L^2(I;W^{2,2}(\Oeta)), \qquad
\eta(0,\cdot)=\eta_0(\cdot)\in W^{5,2}(\omega),
\\&
\partial_t\eta (0,\cdot)=\eta_\star(\cdot)\in W^{3,2}(\omega), \qquad
\bu(0,\cdot)=\bu_0(\cdot)\in W^{3,2}_{\divx}(\Omega_{\eta_0}),
\\&
\text{with }\bu_0 \circ \bm{\varphi}_{\eta_0} =\eta_\star \text{   on } \omega.
\end{align*} 
Suppose further that
\begin{align*}
\eta\in W^{2,2}(I;W^{2,2}(\omega))
\cap
W^{1,\infty}(I;W^{3,2}(\omega))
\cap 
W^{1,2}(I;W^{4,2}(\omega))
\cap
L^\infty(I;W^{5,2}(\omega))
\end{align*}
and  there exists constants $\ell,\kappa_0>0$ such that
\begin{align*}
\Vert\eta\Vert_{L^\infty(I\times\omega)}\leq \ell<L\qquad
\inf_{I\times\omega}\vert\partial_1\bm{\varphi}_\eta\times \partial_2 \bm{\varphi}_\eta\vert \geq \kappa_0.
\end{align*}  
Then there exist a unique solution to \eqref{thm:stokeInhom} such that
\begin{align*}
\int_I\big(\Vert \partial_t\bu\Vert_{W^{2,2}(\Oeta)}^2+&\Vert  \bu\Vert_{W^{4,2}(\Oeta)}^2
+
\Vert p\Vert_{W^{3,2}(\Oeta)}^2\big)\dt
\\
\lesssim&
\Vert\bu_0\Vert_{W^{3,2}(\Omega_{\eta_0})}^2
+
\Vert\eta_\star\Vert_{W^{3,2}(\omega)}^2
+
\int_I\big(\Vert\bff \Vert_{W^{2,2}(\Oeta)}^2
+
 \| \partial_t^2\eta\|_{W^{2,2}( \omega )}^2 
\big)\dt
\\&+
\int_I
\big( 
 \| \naby^2\partial_t \eta \,\partial_t \eta \|_{L^{2}( \omega )}^2
+
\|\partial_t\eta\|_{W^{4,2}( \omega )}^2
+
\|\naby^2\partial_t \eta \,\naby^2 \eta \|_{L^{2}( \omega )}^2
\big)\dt
\\&\qquad+
\int_I
\big(
 \|\naby^3\partial_t \eta\,\naby \eta\|_{L^{2}( \omega )}^2 
+
\|\naby^2\partial_t \eta \,|\naby \eta|^2\|_{L^{2}( \omega )}^2
\big)\dt
\end{align*} 
hold with a constant depending on $\eta_0$.
\end{corollary}

\section{Strong solutions to subproblems}
\label{sec:strongSol}
The initial step in our strategy for actually constructing a solution involves solving the solvent-structure subproblem and the solute subproblem independently of each other. After that, in a subsequent section, we will use a fixed-point argument to get a local solution to the fully coupled system. 
\subsection{The solvent-structure subproblem}
In the following, for a given stress tensor $\underline{\bT}$, a given pair of body forces $\bff$ and $g$, we wish to derive suitable estimates for the unique strong solution of the following solvent-structure system of equations
\begin{align}
\divx \bu=0, \label{divfreeAlone} 
\\
\partial_t \bu  + (\mathbf{u}\cdot \nabx)\mathbf{u} 
= 
 \delx \bu -\nabx p
+\bff
+
\divx   \underline{\bT}, \label{momEqAlone} 
\\
\partial_t^2 \eta - \partial_t\dely \eta + \dely^2 \eta = g - ( \mathbb{S}\bn_\eta )\circ \bm{\varphi}_\eta\cdot\bn \,\vert\det(\naby\bm{\varphi}_\eta)\vert,\label{shellAlone}
\end{align}
defined on $I\times\Oeta\subset \mathbb R^{1+3}$ (with \eqref{shellAlone} defined on   $I\times\omega\subset \mathbb R^{1+2}$) where
\begin{align*}
\mathbb{S}=  \nabx \bu +(\nabx \bu)^\top  -p\mathbb{I}+ \underline{\bT}.
\end{align*}
We then complement \eqref{divfreeAlone}--\eqref{shellAlone} with the following initial and interface conditions
\begin{align}
&\eta(0,\cdot)=\eta_0(\cdot), \qquad\partial_t\eta(0,\cdot)=\eta_\star(\cdot) & \text{in }\omega,  
\\  
&\bu(0,\cdot)=\bu_0(\cdot) & \text{in }\Omega_{\eta_0}.
\\
\label{interfaceAlone}
&\bu\circ\bm{\varphi}_\eta=(\partial_t\eta)\bn & \text{on }I\times \omega.
\end{align} 
We recall that a unique local strong solution to the system above exists as stated in Theorem \ref{thm:BM}. The estimate we wish to derive is required in a subsequent section where we deal with the fully coupled solute-solvent-structure problem. We state this in the following:
\begin{proposition}
\label{prop:BM}
For the dataset $(\bff, g, \eta_0, \eta_\star, \bu_0, \underline{\bT})$ satisfying \eqref{datasetAlone} and with small initial acceleration, that is,  there exists a strictly positive constant $\epsilon<1$ such that
\begin{equation}
\begin{aligned}
\label{smallinitialEnergy}
 c&\big( 
 \Vert \eta_\star\Vert_{W^{1,2}(\omega)}^2
 +
  \Vert \eta_0\Vert_{W^{3,2}(\omega)}^2
  +
  \Vert  \bu_0\Vert_{W^{1,2}(\Omega_{\eta_0})}^2
  \big) 
   \\&\quad+ \int_{I } \big(  \Vert \bff\Vert_{L^2(\Omega_\eta)}^2  +
    \Vert \underline{\bT}\Vert_{W^{1,2}(\Omega_\eta)}^2  +
 \Vert g\Vert_{L^2(\omega)}^2 \big)\dt
 \leq
 \epsilon
\end{aligned}
\end{equation} 
for $c>0$,
let $(\eta,\bu,p)$ be a unique local strong solution of \eqref{divfreeAlone}--\eqref{interfaceAlone}, in the sense of Definition 3.1, on $I_*:=(0,T_*)$, $T_*\leq T$. 
Then the following bound hold
\begin{align*}
\sup_{I_*}
\big(
\Vert   \partial_t^2 \naby \eta\Vert_{L^2(\omega)}^2
&+
\Vert   \partial_t \naby\Dely \eta\Vert_{L^2(\omega)}^2
+
\Vert    \naby\Dely^2 \eta\Vert_{L^2(\omega)}^2
\big) 
 +\int_{I_*}\big(   \Vert\bu\Vert_{W^{4,2}(\Oeta)}^2+\Vert p\Vert_{W^{3,2}(\Oeta)}^2\big)\dt
\\&+
\int_{I_*}\big( \Vert\partial_t^2\bu\Vert_{L^{2}(\Oeta)}^2 + \Vert \partial_t\bu\Vert_{W^{2,2}(\Oeta)}^2+\Vert \partial_tp\Vert_{W^{1,2}(\Oeta)}^2\big)\dt
\\&+
\int_{I_*} 
\big(
\Vert \partial_t \dely^2 \eta\Vert_{L^2(\omega)}^2 
 +
\Vert  \partial_t^2 \dely \eta\Vert_{L^2(\omega)}^2 
+
\Vert  \partial_t^3 \eta\Vert_{L^2(\omega)}^2 
+\Vert \Dely^3\eta\Vert_{L^{2}(\omega)}^2
\big)
\dt
\\&\qquad
\lesssim
\mathcal{E}(\mathrm{data})
+
\int_{I_*} 
\Big( 
\Vert \underline{\bT}
\Vert_{W^{3,2}(\Oeta)}^2
+
\Vert  \partial_t\underline{\bT}
\Vert_{W^{1,2}(\Oeta)}^2
\Big)
\dt,
\end{align*} 
where
\begin{equation}
\begin{aligned}\nonumber
\mathcal{E}(\mathrm{data})
&:=
\Vert \eta_\star\Vert_{W^{3,2}(\omega)}^2
 +
 \Vert \eta_0\Vert_{W^{5,2}(\omega)}^2
  +
 \Vert
  \bu_0\Vert_{W^{3,2}(\Omega_{\eta_0})}^2
   \\&
    +
  \Vert \underline{\bT}(0)
  \Vert_{W^{2,2}(\Omega_{\eta_0})}^2
  +
   \Vert \bff(0)\Vert_{W^{1,2}(\Omega_{\eta_0})}^2    
   +
   \Vert g(0) \Vert_{W^{1,2}(\omega)}^2   
   \\&+
   \int_{I_*} \big(\Vert g \Vert_{W^{2,2}(\omega)}^2
  +
\Vert \partial_t g \Vert_{L^{2}(\omega)}^2
    \big)
\dt
   \\&+
   \int_{I_*} \big( \Vert \bff\Vert_{W^{2,2}(\Oeta)}^2
  +
 \Vert \partial_t\bff\Vert_{L^2(\Oeta)}^2
    \big)
\dt.
\end{aligned}
\end{equation} 
\end{proposition}
\begin{proof}
First of all, if we consider the pair of test functions $(\partial_t\eta, \bu)$ for the shell and momentum equations, respectively, then we obtain the standard energy estimate
\begin{equation}
\begin{aligned}
\label{standEnergy} 
\sup_{I_*}&\big(\Vert \partial_t\eta\|_{L^2(\omega)}^2+ \Vert\Dely\eta\|_{L^2(\omega)}^2
+\Vert\bu\Vert_{L^2(\Oeta)}^2
\big)
+\int_{I_*}\big(\Vert\partial_t\naby\eta\Vert_{L^2(\omega)}^2+\Vert \nabx\bu\Vert_{L^2(\Oeta)}^2 
\big)\dt
\\&\lesssim 
\Vert \bu_0\Vert_{L^{2}(\Omega_{\eta_0})}^2+\Vert \eta_\star \Vert_{L^2(\omega)}^2+\Vert \Dely\eta_0\Vert_{L^{2}(\omega)}^2
+ 
\int_{I_*} \big(  \Vert \bff\Vert_{L^2(\Omega_\eta)}^2  +
    \Vert \underline{\bT}\Vert_{L^{2}(\Omega_\eta)}^2  +
 \Vert g\Vert_{L^2(\omega)}^2 \big)\dt.
\end{aligned} 
\end{equation}
Now, given that  our strong solution as stated in Theorem \ref{thm:BM} is quite regular, by Sobolev embedding, we have in particular that
\begin{align}
\label{serrin}
\bu \in {L^2(I_*;L^\infty(\Omega_\eta))} \quad\text{ and }\quad\eta\in { L^\infty(I_*;C^{1}(\omega))}.
\end{align}
However, as  shown in \cite[Section 4]{BMSS}, if a solution satisfy \eqref{serrin}, then besides the standard energy estimate above, we additionally have the following acceleration estimate 
 \begin{equation}
\begin{aligned}
\label{est:reg}
&\sup_{I_*} 
\big(\Vert \partial_t\naby \eta\Vert_{L^2(\omega)}^2  
+
\Vert \naby\Dely \eta\Vert_{L^2(\omega)}^2  
+
\Vert\nabx  \bu \Vert_{L^2(\Oeta)}^2 
\big)
\\&\quad+
\int_{I_*} 
\big(\Vert \partial_t\Dely \eta \Vert_{L^2(\omega)}^2 + \Vert \partial_t^2 \eta\Vert_{L^2(\omega)}^2 
 +
 \Vert \nabx^2 \bu\Vert_{L^2(\Oeta)}^2  +\Vert \partial_t \bu \Vert_{L^2(\Oeta)}^2   + \Vert \nabx  p\Vert_{L^2(\Oeta)}^2 
 \big) \dt
 \\&\lesssim
 \big( 
 \Vert \eta_\star\Vert_{W^{1,2}(\omega)}^2
 +
  \Vert \eta_0\Vert_{W^{3,2}(\omega)}^2
  +
  \Vert  \bu_0\Vert_{W^{1,2}(\Omega_{\eta_0})}^2
  \big) 
   \\&\quad+ \int_{I_*} \big(  \Vert \bff\Vert_{L^2(\Omega_\eta)}^2  +
    \Vert \underline{\bT}\Vert_{W^{1,2}(\Omega_\eta)}^2  +
 \Vert g\Vert_{L^2(\omega)}^2 \big)\dt.
\end{aligned}
\end{equation}
Furthermore, if we make the highest spatial derivative term in the shell equation the subject, then we obtain
\begin{equation}
\begin{aligned}
\label{highEta}
\int_{I_*}\Vert \Dely^2\eta\Vert_{L^2(\omega)}^2\dt
\lesssim&
\int_{I_*} 
\big(\Vert \partial_t\Dely \eta \Vert_{L^2(\omega)}^2 + \Vert \partial_t^2 \eta\Vert_{L^2(\omega)}^2 
 +
 \Vert g\Vert_{L^2(\omega)}^2 \big)\dt
 \\&+ \int_{I_*}
  \big(\Vert\bu\Vert_{W^{2,2}( \Oeta)}^2
  +
 \Vert p\Vert_{W^{1,2}( \Oeta)}^2
 +
 \Vert \underline{\bT}\Vert_{W^{1,2}( \Oeta)}^2
 \big)\dt
\end{aligned}
\end{equation}
where we have used the trace theorem to obtain the bound
\begin{align*}
\int_{I_*}\int_\omega\vert
  ( \mathbb{S}\bn_\eta )\circ \bm{\varphi}_\eta\cdot\bn \vert^2\,\vert\det(\naby\bm{\varphi}_\eta) 
  \vert^2\dy\dt
  \lesssim&
  \int_{I_*}
  \big(\Vert\bu\Vert_{W^{1,2}(\partial\Oeta)}^2
  +
 \Vert p\Vert_{L^{2}(\partial\Oeta)}^2
 +
 \Vert \underline{\bT}\Vert_{L^{2}(\partial\Oeta)}^2
 \big)\dt
    \\
     \lesssim&
  \int_{I_*}
  \big(\Vert\bu\Vert_{W^{2,2}( \Oeta)}^2
  +
 \Vert p\Vert_{W^{1,2}( \Oeta)}^2
 +
 \Vert \underline{\bT}\Vert_{W^{1,2}( \Oeta)}^2
 \big)\dt.
\end{align*}
Now note that the right-hand side of \eqref{highEta} can be estimated by the right-hand side of \eqref{est:reg}. Therefore, if we combine \eqref{standEnergy}, \eqref{est:reg} and \eqref{highEta}, we obtain the estimate
 \begin{equation}
\begin{aligned}
\label{est:reg1}
&\sup_{I_*} 
\big(\Vert \partial_t  \eta\Vert_{W^{1,2}(\omega)}^2  
+
\Vert \naby\Dely \eta\Vert_{L^2(\omega)}^2  
+
\Vert  \bu \Vert_{W^{1,2}(\Oeta)}^2 
\big)
\\&\quad+
\int_{I_*} 
\big(\Vert  \eta\Vert_{W^{4,2}(\omega)}^2 + \Vert \partial_t  \eta \Vert_{W^{2,2}(\omega)}^2 + \Vert \partial_t^2 \eta\Vert_{L^2(\omega)}^2  
 \big) \dt
 \\&\quad+
\int_{I_*} 
\big( 
 \Vert   \bu\Vert_{W^{2,2}(\Oeta)}^2  +\Vert \partial_t \bu \Vert_{L^2(\Oeta)}^2   + \Vert   p\Vert_{W^{1,2}(\Oeta)}^2 
 \big) \dt
 \\&\lesssim
 \big( 
 \Vert \eta_\star\Vert_{W^{1,2}(\omega)}^2
 +
  \Vert \eta_0\Vert_{W^{3,2}(\omega)}^2
  +
  \Vert  \bu_0\Vert_{W^{1,2}(\Omega_{\eta_0})}^2
  \big) 
   \\&\quad+ \int_{I_* } \big(  \Vert \bff\Vert_{L^2(\Omega_\eta)}^2  +
    \Vert \underline{\bT}\Vert_{W^{1,2}(\Omega_\eta)}^2  +
 \Vert g\Vert_{L^2(\omega)}^2 \big)\dt
  \\&\leq
 \epsilon
\end{aligned}
\end{equation}
where the last inequality follows from  \eqref{smallinitialEnergy}.
In order to obtain higher regularity estimates, we have to take the  time derivative of our original solvent-structure subsystem. This is because of its parabolic nature which  requires one temporal regularity for two extra spatial regularity. The new time-differentiated solvent-structure subsystem becomes
\begin{align} 
\divx \partial_t\bu=0, \label{divfreeAloneTime} 
\\
\partial_t^2\bu  + (\mathbf{u}\cdot \nabx)\partial_t\bu
= 
\delx \partial_t\bu -\nabx\partial_t p  
  +\partial_t\bff - (\partial_t\bu\cdot \nabx)\mathbf{u} 
  +
\divx  \partial_t \underline{\bT}, 
  \label{momEqAloneTime} 
\\
\partial_t^3\eta - \partial_t^2 \dely \eta  + \partial_t\dely^2 \eta = \partial_tg+f_0
+f_1+f_2+f_3
,\label{shellAloneTime}
\end{align}
defined in  $I\times \Oeta$  (with \eqref{shellAloneTime} defined on   $I\times\omega\subset \mathbb R^{1+2}$)  where
\begin{align*}
f_0&:= - (\partial_t \mathbb{S}\bn_\eta )\circ \bm{\varphi}_\eta\cdot\bn \,\vert\det(\naby\bm{\varphi}_\eta)\vert,
\\
f_1&:=- ( \mathbb{S}\partial_t\bn_\eta )\circ \bm{\varphi}_\eta\cdot\bn \,\vert\det(\naby\bm{\varphi}_\eta)\vert,
\\
f_2&:=- \nabx( \mathbb{S}\bn_\eta )\circ\bm{\varphi}_\eta( \partial_t\eta\bn)\cdot\bn \,\vert\det(\naby\bm{\varphi}_\eta)\vert,
\\
f_3&:=- ( \mathbb{S}\bn_\eta )\circ \bm{\varphi}_\eta\cdot\bn \,\partial_t\vert\det(\naby\bm{\varphi}_\eta)\vert,
\\
\partial_t\mathbb{S}&:=  \nabx \partial_t \bu  +(\nabx \partial_t \bu)^\top  - \partial_tp\mathbb{I}+ \partial_t\underline{\bT}.
\end{align*}
For the structure subsystem, we have used that
\begin{align*}
\partial_t\big(( \mathbb{S}\bn_\eta )\circ \bm{\varphi}_\eta\big) 
=&(\partial_t\mathbb{S}\bn_\eta )\circ\bm{\varphi}_\eta
+
( \mathbb{S}\partial_t\bn_\eta )\circ\bm{\varphi}_\eta
+
\nabx( \mathbb{S}\bn_\eta )\circ\bm{\varphi}_\eta( \partial_t\eta\bn)
\end{align*}
where  
\begin{align*} 
\partial_t  \bn_\eta 
&=
 \partial_{y_1} \partial_t\bm{\varphi}_\eta \times  \partial_{y_2} \bm{\varphi}_\eta+ \partial_{y_1} \bm{\varphi}_\eta \times  \partial_{y_2}\partial_t \bm{\varphi}_\eta 
\\&=
 \partial_{y_1} (\partial_t \eta\bn) \times  \partial_{y_2} \bm{\varphi}_\eta+ \partial_{y_1} \bm{\varphi}_\eta \times  \partial_{y_2}(\partial_t \eta\bn).
\end{align*}
The new  subproblem \eqref{divfreeAloneTime}-\eqref{shellAloneTime} will now satisfy a modified interface and initial conditions. We begin by investigating what the new interface condition should look like. For this, we first note  that the chain rule gives
\begin{align*}
\partial_t(\bu\circ\bm{\varphi}_\eta)
=&\partial_t\bu\circ\bm{\varphi}_\eta
+
(\nabx\bu\circ\bm{\varphi}_\eta)^\top \partial_t\bm{\varphi}_\eta. 
\end{align*}
However, we  can use the definition of $\bm{\varphi}_\eta$ (see Section \ref{sec:setup}) and the original interface condition \eqref{interfaceAlone} to get that
\begin{align*}
&\partial_t\bm{\varphi}_\eta
 =( \partial_t\eta)\bn
=\bu\circ\bm{\varphi}_\eta & \text{on }I\times \omega.
\end{align*}
As a result, differentiating the interface condition \eqref{interfaceAlone} leads to
\begin{align} 
\label{interfaceAloneTime}
&\partial_t\bu\circ\bm{\varphi}_\eta
+
((\bu\cdot\nabx)\bu)\circ\bm{\varphi}_\eta=(\partial_t^2 \eta)\bn & \text{on }I\times \omega.
\end{align}
where we have used that
\begin{align*}
(\nabx\bu\circ\bm{\varphi}_\eta)^\top\bu\circ\bm{\varphi}_\eta=
((\nabx\bu)^\top\bu)\circ\bm{\varphi}_\eta
=
((\bu\cdot\nabx)\bu)\circ\bm{\varphi}_\eta.
\end{align*}
Since our new solvent-structure subsystem have been derived from the old by taking the time derivative, the time-differentiated  interface condition \eqref{interfaceAloneTime} above invokes the following compatibility condition  derived from the original solvent-structure  subsystem \eqref{momEqAlone}-\eqref{shellAlone}
\begin{equation}
\begin{aligned}
\label{compatibilityTime}
\big(
 \delx \bu& -\nabx p
+\bff
+
\divx   \underline{\bT}
\big)\circ\bm{\varphi}_\eta
\\
&=\big( \partial_t\dely \eta - \dely^2 \eta+ g - ( \mathbb{S}\bn_\eta )\circ \bm{\varphi}_\eta\cdot\bn \,\vert\det(\naby\bm{\varphi}_\eta)\vert\big)\bn
&\text{ on } \overline{I\times\omega}.
\end{aligned}
\end{equation}
Finally,  the new initial conditions are
\begin{align}
\label{etaStarTime}
&\partial_t\eta 
\big\vert_{t=0}
=\eta_\star
 & \text{in }\omega,  
\\  
& \partial_t^2 \eta 
\big\vert_{t=0}
=\dely \eta_\star - \dely^2 \eta_0 + g(0) - ( \mathbb{S}(0)\bn_{\eta_0} )\circ \bm{\varphi}_{\eta_0}\cdot\bn \,\vert\det(\naby\bm{\varphi}_{\eta_0})\vert
& \text{in }\omega,  
\\  
\label{buTime}
&\partial_t\bu 
\big\vert_{t=0}
=
 \delx \bu_0 -\nabx p(0)
+\bff(0)
+
\divx   \underline{\bT}(0)- (\mathbf{u}_0\cdot \nabx)\mathbf{u}_0
& \text{in }\Omega_{\eta_0}.
\end{align}
Here, the initial pressure $p(0)$ can be recovered by solving the elliptic Robin boundary problem 
 (with boundary condition obtained from the compatibility condition \eqref{compatibilityTime} at $t=0$):
\begin{align*}
&\Delx p(0) 
=  
\divx\bff(0)
+
\divx(\divx   \underline{\bT}(0)-
    \mathbf{u}_0\cdot\nabx\mathbf{u}_0)
&\text{ in } \Omega_{\eta_0} ,
\\
&\bn_{\eta_0}\cdot\nabx p(0)
 +
  p(0)\vert\bn_{\eta_0}  \cdot\bn\circ \bm{\varphi}_{\eta_0}^{-1}\vert^2\vert\det(\naby\bm{\varphi}_{\eta_0})\vert   \circ \bm{\varphi}_{\eta_0}^{-1}
\\
&\qquad\quad=
    (\nabx \bu_0 +(\nabx \bu_0)^\top  ) \bn_{\eta_0}\cdot\bn\circ \bm{\varphi}_{\eta_0}^{-1}(\bn_{\eta_0}\cdot\bn\circ \bm{\varphi}_{\eta_0}^{-1}) \vert\det(\naby\bm{\varphi}_{\eta_0})\vert\circ \bm{\varphi}_{\eta_0}^{-1}
\\
&\qquad\qquad
+
\big(
 \delx \bu_0 
+\bff(0)
+
\divx   \underline{\bT}(0)  \big) \cdot\bn_{\eta_0}
-
\big(( \dely \eta_\star - \dely^2 \eta_0 + g(0) )\bn \big)  \circ \bm{\varphi}_{\eta_0}^{-1}\cdot\bn_{\eta_0}
&\text{ on } \partial\Omega_{\eta_0} .
\end{align*}
Since $\Omega_{\eta_0}$ is of class $C^{1,1}$,  the solution operator to the Robin problem above has the usual properties, i.e. the solution belongs to $W^{k,2}$, $1\leq k\leq2$ if the right-hand side belongs to $ W^{k-2,2}$ and the boundary datum is in $W^{k-3/2,2}$. In particular,
\begin{equation}
\begin{aligned}
\label{initialPressEst}
\Vert p(0)\Vert_{W^{2,2}(\Omega_{\eta_0})}^2
\lesssim&
\Vert \eta_\star\Vert_{W^{3,2}(\omega)}^2
 +
 \Vert \eta_0\Vert_{W^{5,2}(\omega)}^2  
     +
  \Vert g(0) \Vert_{W^{1,2}(\omega)}^2  
   \\&
    +
  \Vert \underline{\bT}(0)
  \Vert_{W^{2,2}(\Omega_{\eta_0})}^2
  +
   \Vert \bff(0)\Vert_{W^{1,2}(\Omega_{\eta_0})}^2 
  +
  (1+\epsilon)
 \Vert
  \bu_0\Vert_{W^{3,2}(\Omega_{\eta_0})}^2 
\end{aligned}
\end{equation}
 where we have used that
 \begin{align*}
\Vert \divx(\bu_0\cdot\nabx\bu_0)\Vert_{L^2(\Omega_{\eta_0})}^2
&\lesssim
\Vert  \bu_0 \Vert_{L^4(\Omega_{\eta_0})}^2
\Vert  \bu_0 \Vert_{W^{2,4}(\Omega_{\eta_0})}^2
+
\Vert  \bu_0 \Vert_{W^{1,2}(\Omega_{\eta_0})}^2
\Vert  \bu_0 \Vert_{W^{1,\infty}(\Omega_{\eta_0})}^2
\\
&\lesssim 
\Vert  \bu_0 \Vert_{W^{1,2}(\Omega_{\eta_0})}^2
\Vert  \bu_0 \Vert_{W^{3,2}(\Omega_{\eta_0})}^2
\\
&\leq 
\epsilon
\Vert  \bu_0 \Vert_{W^{3,2}(\Omega_{\eta_0})}^2.
 \end{align*}
With the setup above, we can now proceed to obtain our desired estimate for the time-differentiated solvent-structure subsystem \eqref{divfreeAloneTime}-\eqref{buTime}. For this, we first observe that by using \eqref{est:reg1},
\begin{align*}
\Vert\det(\naby\bm{\varphi}_\eta)\Vert_{L^\infty(\omega)}^2
&\lesssim \Vert \bm{\varphi}\Vert_{W^{1,\infty}(\omega)}^4+\Vert\eta\Vert_{W^{3,2}(\omega)}^2+\Vert\eta\Vert_{W^{3,2}(\omega)}^4
\leq\epsilon,
\\
\Vert\partial_t\bn_\eta\Vert_{L^2(\omega)}^2
&\lesssim
\Vert\partial_t\eta\Vert_{W^{1,2}(\omega)}^2\big(1+\Vert\eta\Vert_{W^{3,2}(\omega)}^2\big)\leq\epsilon,
\\
\Vert \bn_\eta\Vert_{W^{1,4}(\omega)}^2
&\lesssim
\Vert \bm{\varphi}\Vert_{W^{1,\infty}(\omega)}+
\Vert \eta\Vert_{W^{3,2}(\omega)}^2\big(1+\Vert\eta\Vert_{W^{3,2}(\omega)}^2\big)\leq\epsilon,
\\
\Vert\partial_t \eta\Vert_{L^4(\omega)}^2
&\lesssim
\Vert\partial_t\eta\Vert_{W^{1,2}(\omega)}^2 \leq\epsilon,
\end{align*}
and by the trace theorem
\begin{align*}
\Vert\partial_t\mathbb{S}\Vert_{L^2(\partial\Omega_\eta)}^2
&\lesssim  
\Vert\partial_t \bu\Vert_{W^{2,2}(\Omega_\eta)}^2  +\Vert \partial_tp\Vert_{W^{1,2}(\Omega_\eta)}^2 +
\Vert \partial_t\underline{\bT}\Vert_{W^{1,2}(\Omega_\eta)}^2 ,
\\
\Vert \mathbb{S}\Vert_{W^{1,\infty}(\partial\Omega_\eta)}^2
&\lesssim  
\Vert  \bu\Vert_{W^{4,2}(\Omega_\eta)}^2  +\Vert  p\Vert_{W^{3,2}(\Omega_\eta)}^2 +
\Vert  \underline{\bT}\Vert_{W^{3,2}(\Omega_\eta)}^2 .
\end{align*}
As such,
\begin{equation}
\begin{aligned}\label{f0f3}
\Vert f_0\Vert_{L^2(\omega)}^2
&\leq
\epsilon\big(\Vert\partial_t \bu\Vert_{W^{2,2}(\Omega_\eta)}^2  +\Vert \partial_tp\Vert_{W^{1,2}(\Omega_\eta)}^2 +
\Vert \partial_t\underline{\bT}\Vert_{W^{1,2}(\Omega_\eta)}^2 
\big),
\\
\Vert f_1\Vert_{L^2(\omega)}^2
&\leq
\epsilon\big(\Vert  \bu\Vert_{W^{4,2}(\Omega_\eta)}^2  +\Vert  p\Vert_{W^{3,2}(\Omega_\eta)}^2 +
\Vert  \underline{\bT}\Vert_{W^{3,2}(\Omega_\eta)}^2
\big),
\\
\Vert f_2\Vert_{L^2(\omega)}^2
&\leq
\epsilon\big(\Vert  \bu\Vert_{W^{4,2}(\Omega_\eta)}^2  +\Vert  p\Vert_{W^{3,2}(\Omega_\eta)}^2 +
\Vert  \underline{\bT}\Vert_{W^{3,2}(\Omega_\eta)}^2
\big),
\\
\Vert f_3\Vert_{L^2(\omega)}^2
&\leq
\epsilon\big(\Vert  \bu\Vert_{W^{4,2}(\Omega_\eta)}^2  +\Vert  p\Vert_{W^{3,2}(\Omega_\eta)}^2 +
\Vert  \underline{\bT}\Vert_{W^{3,2}(\Omega_\eta)}^2
\big) .
\end{aligned}
\end{equation}
If we now test \eqref{shellAloneTime} with $ - \partial_t^2 \dely \eta$ and use \eqref{f0f3}, we obtain
\begin{align*}
\sup_{I_*}
\big(
\Vert   \partial_t^2 \naby \eta\Vert_{L^2(\omega)}^2
&+
\Vert   \partial_t \naby\Dely \eta\Vert_{L^2(\omega)}^2
\big) 
+
\int_{I_*} 
\Vert  \partial_t^2 \dely \eta\Vert_{L^2(\omega)}^2 
\dt
\\&\leq
c\Vert   (\partial_t^2 \naby \eta)(0)\Vert_{L^2(\omega)}^2
+
\Vert   (\partial_t \naby\Dely \eta)(0)\Vert_{L^2(\omega)}^2
+
c\int_{I_*}\Vert\partial_tg\Vert_{L^2(\omega)}^2\dt
\\&
+
\epsilon\int_{I_*}
\big(\Vert\partial_t \bu\Vert_{W^{2,2}(\Omega_\eta)}^2  +\Vert \partial_tp\Vert_{W^{1,2}(\Omega_\eta)}^2 +
\Vert \partial_t\underline{\bT}\Vert_{W^{1,2}(\Omega_\eta)}^2 
\big)\dt
\\
&+
\epsilon\int_{I_*}
\big(\Vert  \bu\Vert_{W^{4,2}(\Omega_\eta)}^2  +\Vert  p\Vert_{W^{3,2}(\Omega_\eta)}^2 +
\Vert  \underline{\bT}\Vert_{W^{3,2}(\Omega_\eta)}^2
\big)\dt
\end{align*}
where using \eqref{etaStarTime}-\eqref{initialPressEst} yield
\begin{align*}
\Vert   (\partial_t^2 \naby \eta)(0)\Vert_{L^2(\omega)}^2
+
\Vert   (\partial_t \naby\Dely \eta)(0)\Vert_{L^2(\omega)}^2
\lesssim&
\Vert \eta_\star\Vert_{W^{3,2}(\omega)}^2
 +
 \Vert \eta_0\Vert_{W^{5,2}(\omega)}^2
  +
  (1+\epsilon)
 \Vert
  \bu_0\Vert_{W^{3,2}(\Omega_{\eta_0})}^2
   \\&
    +
  \Vert \underline{\bT}(0)
  \Vert_{W^{2,2}(\Omega_{\eta_0})}^2
  +
   \Vert \bff(0)\Vert_{W^{1,2}(\Omega_{\eta_0})}^2   
     +
  \Vert g(0) \Vert_{W^{1,2}(\omega)}^2.   
\end{align*}
Consequently, by defining
\begin{align*}
\mathcal{D}_1:=&
c\Vert \eta_\star\Vert_{W^{3,2}(\omega)}^2
 +
 c\Vert \eta_0\Vert_{W^{5,2}(\omega)}^2
  +
  (1+\epsilon)
 \Vert
  \bu_0\Vert_{W^{3,2}(\Omega_{\eta_0})}^2
   \\&
    +
 c \Vert \underline{\bT}(0)
  \Vert_{W^{2,2}(\Omega_{\eta_0})}^2
  +
   c\Vert \bff(0)\Vert_{W^{1,2}(\Omega_{\eta_0})}^2   
     +
  c\Vert g(0) \Vert_{W^{1,2}(\omega)}^2  
   +
   c
   \int_{I_*}  
  \Vert\partial_tg\Vert_{L^2(\omega)}^2 
\dt
\\&
+
\epsilon\int_{I_*}
\big(\Vert\partial_t \bu\Vert_{W^{2,2}(\Omega_\eta)}^2  +\Vert \partial_tp\Vert_{W^{1,2}(\Omega_\eta)}^2 +
\Vert \partial_t\underline{\bT}\Vert_{W^{1,2}(\Omega_\eta)}^2 
\big)\dt
\\
&+
\epsilon\int_{I_*}
\big(\Vert  \bu\Vert_{W^{4,2}(\Omega_\eta)}^2  +\Vert  p\Vert_{W^{3,2}(\Omega_\eta)}^2 +
\Vert  \underline{\bT}\Vert_{W^{3,2}(\Omega_\eta)}^2
\big)\dt,
\end{align*}
then we have that
\begin{align}
\label{d1}
\sup_{I_*}
\big(
\Vert   \partial_t^2 \naby \eta\Vert_{L^2(\omega)}^2
&+
\Vert   \partial_t \naby\Dely \eta\Vert_{L^2(\omega)}^2
\big) 
+
\int_{I_*} 
\Vert  \partial_t^2 \dely \eta\Vert_{L^2(\omega)}^2 
\dt
\leq \mathcal{D}_1.
\end{align}
Next, if we test \eqref{shellAloneTime} with $\partial_t \dely^2 \eta$ and use \eqref{f0f3}, we obtain
\begin{align}
\label{d2}
\sup_{I_*}
\Vert   \partial_t \naby\Dely \eta\Vert_{L^2(\omega)}^2 
+&
\int_{I_*} 
\Vert  \partial_t \dely^2 \eta\Vert_{L^2(\omega)}^2 
\dt
 \leq
c\int_{I_*}\Vert\partial_t^3\eta\Vert_{L^2(\omega)}^2\dt
+
\mathcal{D}_1.
\end{align}
To estimate the first term on the right-hand side above, we test \eqref{shellAloneTime} with $\partial_t^3 \eta$ and use \eqref{f0f3} to obtain  
\begin{align*}
\sup_{I_*}
\Vert   \partial_t^2 \naby  \eta\Vert_{L^2(\omega)}^2 
+&
\int_{I_*} 
\Vert  \partial_t^3 \eta\Vert_{L^2(\omega)}^2 
\dt
 \leq
c\int_{I_*} \Vert\partial_t^2\dely\eta\Vert_{L^2(\omega)}^2\dt
+
c\int_{I_*}\frac{\dd}{\dt}\int_\omega\partial_t\naby\dely\eta\,\partial_t^2\naby\eta\dy\dt
+
\mathcal{D}_1
\end{align*}
where using \eqref{d1} and Young's inequality yield
\begin{align*}
c\int_{I_*} \Vert\partial_t^2\dely\eta\Vert_{L^2(\omega)}^2\dt+
c\int_{I_*}\frac{\dd}{\dt}\int_\omega\partial_t\naby\dely\eta\,\partial_t^2\naby\eta\dy\dt
\leq \mathcal{D}_1.
\end{align*} 
Thus,
\begin{align}
\label{d4}
\sup_{I_*}
\Vert   \partial_t^2 \naby  \eta\Vert_{L^2(\omega)}^2 
+&
\int_{I_*} 
\Vert  \partial_t^3 \eta\Vert_{L^2(\omega)}^2 
\dt
 \leq 
\mathcal{D}_1.
\end{align}
We can now combine \eqref{d1}, \eqref{d2} and \eqref{d4} and conclude that
\begin{align}
\label{d4}
\sup_{I_*}
\big(
\Vert   \partial_t^2 \naby \eta\Vert_{L^2(\omega)}^2
&+
\Vert   \partial_t \naby\Dely \eta\Vert_{L^2(\omega)}^2
\big) 
+
\int_{I_*} 
\big(
\Vert \partial_t \dely^2 \eta\Vert_{L^2(\omega)}^2 
 +
\Vert  \partial_t^2 \dely \eta\Vert_{L^2(\omega)}^2 
+
\Vert  \partial_t^3 \eta\Vert_{L^2(\omega)}^2 
\big)
\dt
\leq \mathcal{D}_1.
\end{align}
This completes the desired estimate for the structure subsystem so we now tackle the solvent subsystem.
To directly obtain higher order regularity estimate for the fluid's velocity and pressure, we rewrite the momentum equation as 
\begin{align*}
&\partial_t \bu  -
 \delx \bu +\nabx p
=\bff
+
\divx   \underline{\bT}- (\mathbf{u}\cdot \nabx)\mathbf{u} &\text{ in } I_*\times\Oeta
\end{align*}
with $\bu\circ\bm{\varphi}_\eta=(\partial_t\eta)\bn$ on $I_*\times \omega$. Then, by maximal regularity, Corollary \ref{cor:stokeInhom},
\begin{align*}
\int_{I_*}\big(\Vert \partial_t\bu\Vert_{W^{2,2}(\Oeta)}^2+&\Vert  \bu\Vert_{W^{4,2}(\Oeta)}^2
+
\Vert p\Vert_{W^{3,2}(\Oeta)}^2\big)\dt
\\
\lesssim&
\Vert\bu_0\Vert_{W^{3,2}(\Omega_{\eta_0})}^2
+
\Vert\eta_\star\Vert_{W^{3,2}(\omega)}^2
+
\int_{I_*}\big(\Vert\bff \Vert_{W^{2,2}(\Oeta)}^2
+
\Vert  \underline{\bT}\Vert_{W^{3,2}(\Oeta)}^2
\big)\dt
\\& 
+ 
\int_{I_*}\big(\Vert \bu\cdot\nabx\bu\Vert_{W^{2,2}(\Oeta)}^2
+
 \| \partial_t^2\eta\|_{W^{2,2}( \omega )}^2 \big)\dt
\\&+
\int_{I_*}
\big( 
 \| \naby^2\partial_t \eta \,\partial_t \eta \|_{L^{2}( \omega )}^2
+
\|\partial_t\eta\|_{W^{4,2}( \omega )}^2
+
\|\naby^2\partial_t \eta \,\naby^2 \eta \|_{L^{2}( \omega )}^2
\big)\dt
\\&+
\int_{I_*}
\big(
 \|\naby^3\partial_t \eta\,\naby \eta\|_{L^{2}( \omega )}^2 
+
\|\naby^2\partial_t \eta \,|\naby \eta|^2\|_{L^{2}( \omega )}^2
\big)\dt
\end{align*}
where using the Banach algebraic property of $W^{2,2}$, interpolation and \eqref{est:reg1} yield
\begin{align*}
\int_{I_*}\Vert \bu\cdot\nabx\bu\Vert_{W^{2,2}(\Oeta)}^2\dt
&\lesssim
\int_{I_*}\Vert \bu \Vert_{W^{1,2}(\Oeta)}^{4/3}
\Vert \bu \Vert_{W^{4,2}(\Oeta)}^{2/3}
\Vert \bu \Vert_{W^{1,2}(\Oeta)}^{2/3}
\Vert \bu \Vert_{W^{4,2}(\Oeta)}^{4/3}\dt
\\&\lesssim
\int_{I_*}\Vert \bu \Vert_{W^{1,2}(\Oeta)}^2
\Vert \bu \Vert_{W^{4,2}(\Oeta)}^2\dt
\\&\leq
\epsilon
\int_{I_*}
\Vert \bu \Vert_{W^{4,2}(\Oeta)}^2\dt.
\end{align*}
We also have
\begin{align*}
\int_{I_*} 
 \| \naby^2\partial_t \eta \,\partial_t \eta \|_{L^{2}( \omega )}^2\dt
 &\lesssim
 \sup_{I_*}
 \Vert \partial_t \eta\Vert_{W^{1,2}(\omega)}^2
 \int_{I_*} 
 \| \partial_t \eta   \|_{W^{3,2}( \omega )}^2\dt
 \\&\leq
 \epsilon
 \int_{I_*} 
\Vert \partial_t\Dely^2\eta \Vert_{L^{2}(\omega)}^2  \dt
\end{align*}
and
\begin{align*}
\int_{I_*}
\big(
\|\naby^2\partial_t \eta \,\naby^2 \eta \|_{L^{2}( \omega )}^2
+&
\|\naby^3\partial_t \eta\,\naby \eta\|_{L^{2}( \omega )}^2 
\big)
\dt
\\&\lesssim
 \sup_{I_*}
 \Vert  \eta\Vert_{W^{3,2}(\omega)}^2
 \int_{I_*} 
 \| \partial_t \eta   \|_{W^{3,2}( \omega )}^2\dt
 \\&\leq
 \epsilon
 \int_{I_*} 
\Vert \partial_t\Dely^2\eta \Vert_{L^{2}(\omega)}^2  \dt
\end{align*}
as well as
\begin{align*}
\int_{I_*}
\|\naby^2\partial_t \eta \,|\naby \eta|^2\|_{L^{2}( \omega )}^2
\dt
&\lesssim
 \Big(\sup_{I_*}
 \Vert  \eta\Vert_{W^{3,2}(\omega)}^2
 \Big)^2
 \int_{I_*} 
 \| \partial_t \eta   \|_{W^{2,2}( \omega )}^2\dt
 \\&\leq
 \epsilon
 \int_{I_*} 
\Vert \partial_t\Dely^2\eta \Vert_{L^{2}(\omega)}^2  \dt.
\end{align*} 
We can, therefore, conclude that
\begin{equation}
\begin{aligned}\label{mR1}
\int_{I_*}\big( \Vert\partial_t\bu\Vert_{W^{2,2}(\Oeta)}^2 +& \Vert\bu\Vert_{W^{4,2}(\Oeta)}^2+\Vert p\Vert_{W^{3,2}(\Oeta)}^2\big)\dt
\\
\lesssim&
\Vert\bu_0\Vert_{W^{3,2}(\Omega_{\eta_0})}^2
+
\Vert\eta_\star\Vert_{W^{3,2}(\omega)}^2
+
\int_{I_*}\big( \Vert\bff\Vert_{W^{2,2}(\Oeta)}^2+\Vert  \underline{\bT}\Vert_{W^{3,2}(\Oeta)}^2  \big)\dt
\\&
+
\int_{I_*}\big(
\Vert \partial_t^2\Dely\eta \Vert_{L^{2}(\omega)}^2 
+
\Vert \partial_t\Dely^2\eta \Vert_{L^{2}(\omega)}^2 \big)\dt.
\end{aligned}
\end{equation}  
Next, we apply Theorem \ref{thm:stokeInhom} to the system
 \begin{align*}
 &\partial_t^2\bu  -\delx \partial_t\bu +\nabx\partial_t p  
=\partial_t\bff +
\divx  \partial_t \underline{\bT}- (\partial_t\bu\cdot \nabx)\mathbf{u} 
 -(\mathbf{u}\cdot \nabx)\partial_t\bu   &\text{ in } I_*\times\Oeta,
\\
&\divx\partial_t\bu=0
&\text{ in } I_*\times\Oeta,
\\&
\partial_t\bu\circ\bm{\varphi}_\eta
=(\partial_t^2 \eta-(\naby\bu\circ\bm{\varphi}_\eta)^\top\partial_t \eta)\bn
&\text{ on } I_*\times \omega 
\end{align*}
with the initial condition \eqref{buTime}. We note that this interface condition above is equivalent to \eqref{interfaceAloneTime}. By Theorem \ref{thm:stokeInhom} and \eqref{initialPressEst}, it follows that
\begin{align*}
\int_{I_*}\big( \Vert\partial_t^2\bu\Vert_{L^{2}(\Oeta)}^2 &+ \Vert \partial_t\bu\Vert_{W^{2,2}(\Oeta)}^2+\Vert \partial_tp\Vert_{W^{1,2}(\Oeta)}^2\big)\dt
\\
\lesssim&
\Vert \eta_\star\Vert_{W^{3,2}(\omega)}^2
 +
 \Vert \eta_0\Vert_{W^{5,2}(\omega)}^2  
     +
  \Vert g(0) \Vert_{W^{1,2}(\omega)}^2  
   \\&
    +
  \Vert \underline{\bT}(0)
  \Vert_{W^{2,2}(\Omega_{\eta_0})}^2
  +
   \Vert \bff(0)\Vert_{W^{1,2}(\Omega_{\eta_0})}^2 
  +
  (1+\epsilon)
 \Vert
  \bu_0\Vert_{W^{3,2}(\Omega_{\eta_0})}^2 
\\&+
\int_{I_*}\big( \Vert\partial_t\bff\Vert_{L^{2}(\Oeta)}^2+\Vert \partial_t \underline{\bT}\Vert_{W^{1,2}(\Oeta)}^2\big)\dt
\\& 
+ 
\int_{I_*}\big(\Vert \partial_t\bu\cdot\nabx\bu\Vert_{L^{2}(\Oeta)}^2
+\Vert \bu\cdot\nabx\partial_t\bu\Vert_{L^{2}(\Oeta)}^2\big)\dt 
\\&
+
\int_{I_*}\big( 
 \|\partial_t^3\eta \|_{L^{2}( \omega )}^2 
+
\|\partial_t^2\eta \|_{W^{2,2}( \omega )}^2
 +
  \| \partial_t^2 \eta\,\partial_t \eta   \|_{L^{2}( \omega )}^2
+
\|\partial_t^2\eta\,\naby^2 \eta \|_{L^{2}( \omega )}^2
\big)\dt
\\&
+
\int_{I_*}\big( 
 \|\naby\partial_t^2 \eta\,\naby \eta\|_{L^{2}( \omega )}^2 
 +
 \|\partial_t^2\eta\,|\naby \eta|^2 \|_{L^{2}( \omega )}^2
  +
\|(\naby\bu\circ\bm{\varphi}_\eta)^\top|\partial_t \eta|^2 \|_{L^{2}( \omega )}^2 
\big)\dt
\\&+
\int_{I_*}
\big( 
\|(\naby\bu\circ\bm{\varphi}_\eta)^\top\partial_t\eta \|_{W^{2,2}( \omega )}^2 
+
 \|(\naby\bu\circ\bm{\varphi}_\eta)^\top\partial_t^2\eta \|_{L^{2}( \omega )}^2 
\big)\dt
\\&+
\int_{I_*}
\big( 
 \|  (\naby\partial_t\bu\circ\bm{\varphi}_\eta)^\top\partial_t\eta \|_{L^{2}( \omega )}^2  
+
\| (\naby\bu\circ\bm{\varphi}_\eta)^\top\partial_t\eta\,\naby^2 \eta \|_{L^{2}( \omega )}^2
\big)\dt
\\&+
\int_{I_*}
\big(
 \|\naby(\naby\bu\circ\bm{\varphi}_\eta)^\top\partial_t \eta\,\naby \eta\|_{L^{2}( \omega )}^2 
  +
 \|(\naby\bu\circ\bm{\varphi}_\eta)^\top\naby\partial_t \eta\,\naby \eta\|_{L^{2}( \omega )}^2 
\big)\dt
\\&+
\int_{I_*}
\big( 
\|(\naby\bu\circ\bm{\varphi}_\eta)^\top\partial_t\eta\,|\naby \eta|^2 \|_{L^{2}( \omega )}^2
\big)\dt
\end{align*}
where we can use \eqref{est:reg1} to obtain
\begin{align*}
\int_{I_*}\big(\Vert \partial_t\bu\cdot\nabx\bu\Vert_{L^{2}(\Oeta)}^2
+\Vert \bu\cdot\nabx\partial_t\bu\Vert_{L^{2}(\Oeta)}^2\big)\dt
&\lesssim
\int_{I_*} \Vert \partial_t\bu\Vert_{W^{2,2}(\Oeta)}^2
\Vert \bu\Vert_{W^{1,2}(\Oeta)}^2
 \dt
\\&\leq
\epsilon
\int_{I_*}
\Vert \partial_t\bu \Vert_{W^{2,2}(\Oeta)}^2\dt.
\end{align*} 
We also have by using \eqref{est:reg1} again
\begin{align*} 
\int_{I_*}
\big(  
\|\partial_t^2\eta\,\partial_t \eta  \|_{L^{2}( \omega )}^2 
+&
\|\partial_t^2\eta\,\naby^2 \eta \|_{L^{2}( \omega )}^2 
+
 \|\naby\partial_t^2 \eta\,\naby \eta\|_{L^{2}( \omega )}^2   
\big)\dt
\\&\lesssim
\Big(\sup_{I_*}\Vert\partial_t\eta\Vert_{W^{1,2}(\omega)}^2 
+
 \sup_{I_*}\Vert\eta\Vert_{W^{3,2}(\omega)}^2 
 \Big)
\int_{I_*}\Vert\partial_t^2\eta\Vert_{W^{2,2}(\omega)}^2 \dt
\\&\leq
\epsilon
\int_{I_*}
\|\partial_t^2\Dely \eta \|_{L^{2}( \omega )}^2 \dt
\end{align*}
and
\begin{align*} 
\int_{I_*} 
\|\partial_t^2\eta\,|\naby \eta|^2 \|_{L^{2}( \omega )}^2 
 \dt
 &\lesssim
 \Big(\sup_{I_*}\Vert\eta\Vert_{W^{3,2}(\omega)}^2 \Big)^2
\int_{I_*}\Vert\partial_t^2\eta\Vert_{W^{2,2}(\omega)}^2 \dt
\\&\leq
\epsilon
\int_{I_*}
\|\partial_t^2\Dely \eta \|_{L^{2}( \omega )}^2 \dt.
\end{align*}
Furthermore, by trace theorem and \eqref{est:reg1}
\begin{align*} 
\int_{I_*} 
\|(\naby\bu\circ\bm{\varphi}_\eta)^\top|\partial_t \eta|^2 \|_{L^{2}( \omega )}^2 
\dt 
&\lesssim
\Big(
\sup_{I_*}\Vert\partial_t\eta\Vert_{W^{1,2}(\omega)}^2
\Big)^2
\int_{I_*}
\Vert \bu \Vert_{W^{1,6}(\partial\Oeta)}^2 \dt
\\&\leq
\epsilon
\int_{I_*}
\Vert \bu \Vert_{W^{4,2}(\Oeta)}^2\dt.
\end{align*}
If we now use that $W^{2,2}$ is a Banach algebra and that the embedding
\begin{align*}
W^{1,2}(I_*;W^{2,2}(\Oeta))\cap L^{2}(I_*;W^{4,2}(\Oeta))
\hookrightarrow
L^{\infty}(I_*;W^{7/2,2}(\Oeta))
\end{align*}
is continuous, we obtain
\begin{align*}
\int_{I_*}\|(\naby\bu\circ\bm{\varphi}_\eta)^\top\partial_t\eta \|_{W^{2,2}( \omega )}^2 \dt
&\lesssim 
\int_{I_*}\Vert \bu \Vert_{W^{3,2}(\partial\Oeta)}^2\| \partial_t\eta \|_{W^{2,2}( \omega )}^2 \dt
\\&
\leq \epsilon
\sup_{I_*}\Vert \bu \Vert_{W^{7/2,2}(\Oeta)}^2 
\\&
\leq \epsilon
\int_{I_*}\Vert\partial_t \bu \Vert_{W^{2,2}(\Oeta)}^2\dt +
\epsilon
\int_{I_*}\Vert \bu \Vert_{W^{4,2}(\Oeta)}^2\dt. 
\end{align*}
Next, we note that
\begin{align*}
\int_{I_*} \|(\naby\bu\circ\bm{\varphi}_\eta)^\top\partial_t^2\eta \|_{L^{2}( \omega )}^2 \dt
&\lesssim
\sup_{I_*}\Vert \partial_t^2\eta\Vert_{L^4(\omega)}^2
\int_{I_*} \| \bu  \|_{W^{1,4}(\partial \Oeta )}^2 \dt
\\
&\lesssim
\sup_{I_*}\Vert \partial_t^2\eta\Vert_{W^{1,2}(\omega)}^2
\int_{I_*} \| \bu  \|_{W^{2,2}(  \Oeta )}^2 \dt
\\&\leq
\epsilon
\sup_{I_*}\Vert \partial_t^2\naby\eta\Vert_{L^{2}(\omega)}^2
\end{align*}
and
\begin{align*} 
\int_{I_*} \|  \naby\partial_t\bu\circ\bm{\varphi}_\eta \partial_t\eta \|_{L^{2}( \omega )}^2 \dt
&\lesssim 
\sup_{I_*}\Vert\partial_t\eta\Vert_{W^{1,2}(\omega)}^2 
\int_{I_*}
\Vert\partial_t \bu \Vert_{W^{1,4}(\partial\Oeta)}^2 \dt
\\&\leq
\epsilon
\int_{I_*}
\Vert \partial_t\bu \Vert_{W^{2,2}(\Oeta)}^2\dt.
\end{align*}
Moreover,
\begin{align*} 
\int_{I_*} 
\|(\naby\bu\circ\bm{\varphi}_\eta)^\top\partial_t\eta\,\naby^2 \eta \|_{L^{2}( \omega )}^2
\dt 
&\lesssim
 \sup_{I_*}\Vert\eta\Vert_{W^{3,2}(\omega)}^2 
\sup_{I_*}\Vert\partial_t\eta\Vert_{W^{1,2}(\omega)}^2
\int_{I_*}
\Vert \bu \Vert_{W^{1,6}(\partial\Oeta)}^2 \dt
\\&\leq
\epsilon
\int_{I_*}
\Vert \bu \Vert_{W^{4,2}(\Oeta)}^2\dt
\end{align*}
and similarly,
\begin{align*}
\int_{I_*}
\big( 
 \|\naby(\naby\bu\circ\bm{\varphi}_\eta)^\top\partial_t \eta\,\naby \eta\|_{L^{2}( \omega )}^2  
\big)\dt
&\lesssim
 \sup_{I_*}\Vert\eta\Vert_{W^{3,2}(\omega)}^2 
\sup_{I_*}\Vert\partial_t\eta\Vert_{W^{1,2}(\omega)}^2
\int_{I_*}
\Vert \bu \Vert_{W^{2,4}(\partial\Oeta)}^2 \dt
\\&\leq
\epsilon
\int_{I_*}
\Vert \bu \Vert_{W^{4,2}(\Oeta)}^2\dt.
\end{align*}
We also have
\begin{align*} 
\int_{I_*} 
 \|(\naby\bu\circ\bm{\varphi}_\eta)^\top\naby\partial_t \eta\,\naby \eta\|_{L^{2}( \omega )}^2 
 \dt 
&\lesssim
 \sup_{I_*}\Vert\eta\Vert_{W^{3,2}(\omega)}^2 
\sup_{I_*}\Vert\partial_t\eta\Vert_{W^{1,2}(\omega)}^2
\int_{I_*}
\Vert \bu \Vert_{W^{1,\infty}(\partial\Oeta)}^2 \dt
\\&\leq
\epsilon
\int_{I_*}
\Vert \bu \Vert_{W^{4,2}(\Oeta)}^2\dt
\end{align*}
as well as
\begin{align*}
\int_{I_*}
\big(  
\|(\naby\bu\circ\bm{\varphi}_\eta)^\top\partial_t\eta\,|\naby \eta|^2 \|_{L^{2}( \omega )}^2
\big)\dt
&\lesssim
\Big(\sup_{I_*}\Vert\eta\Vert_{W^{3,2}(\omega)}^2\Big)^2
\sup_{I_*}\Vert\partial_t\eta\Vert_{W^{1,2}(\omega)}^2
\int_{I_*}
\Vert \bu \Vert_{W^{1,4}(\partial\Oeta)}^2 \dt
\\&\leq
\epsilon
\int_{I_*}
\Vert \bu \Vert_{W^{4,2}(\Oeta)}^2\dt.
\end{align*}
Collecting the estimates above while absorbing any $ \epsilon
\int_{I_*}\Vert\partial_t \bu \Vert_{W^{2,2}(\Oeta)}^2\dt $-term on the right-hand side into the left-hand side, we conclude that
\begin{equation}
\begin{aligned}\label{mR2}
\int_{I_*}\big( \Vert\partial_t^2\bu\Vert_{L^{2}(\Oeta)}^2 &+ \Vert \partial_t\bu\Vert_{W^{2,2}(\Oeta)}^2+\Vert \partial_tp\Vert_{W^{1,2}(\Oeta)}^2\big)\dt
\\
\lesssim&
\Vert \eta_\star\Vert_{W^{3,2}(\omega)}^2
 +
 \Vert \eta_0\Vert_{W^{5,2}(\omega)}^2  
     +
  \Vert g(0) \Vert_{W^{1,2}(\omega)}^2  
   \\&
    +
  \Vert \underline{\bT}(0)
  \Vert_{W^{2,2}(\Omega_{\eta_0})}^2
  +
   \Vert \bff(0)\Vert_{W^{1,2}(\Omega_{\eta_0})}^2 
  + 
 \Vert
  \bu_0\Vert_{W^{3,2}(\Omega_{\eta_0})}^2 
\\&+
\int_{I_*}\big( \Vert\partial_t\bff\Vert_{L^{2}(\Oeta)}^2+\Vert \partial_t \underline{\bT}\Vert_{W^{1,2}(\Oeta)}^2+
 \|\partial_t^3\eta \|_{L^{2}( \omega )}^2 \big)\dt  
\\&
+  
\int_{I_*}
\|\partial_t^2\Dely \eta \|_{L^{2}( \omega )}^2 \dt 
+
\epsilon
\int_{I_*}\Vert \bu \Vert_{W^{4,2}(\Oeta)}^2\dt
 +
\epsilon
\sup_{I_*}\Vert \partial_t^2\naby\eta\Vert_{L^{2}(\omega)}^2.
\end{aligned}
\end{equation} 
If we now make the highest spatial derivative term in \eqref{shellAlone} the subject and use trace theorem, we obtain
\begin{equation}
\begin{aligned}
\label{mR2x}    
\int_{I_*}\Vert \Dely^3\eta\Vert_{L^{2}(\omega)}^2
\lesssim&
\int_{I_*}\big(   \Vert\bu\Vert_{W^{4,2}(\Oeta)}^2+\Vert p\Vert_{W^{3,2}(\Oeta)}^2\big)\dt
\\&+
\int_{I_*} 
\big(
\Vert \partial_t \dely^2 \eta\Vert_{L^2(\omega)}^2 
 +
\Vert  \partial_t^2 \dely \eta\Vert_{L^2(\omega)}^2  
\big)
\dt
\\&+
\int_{I_*} 
\big(
\Vert g\Vert_{W^{2,2}(\omega)}^2 
 +
\Vert  \underline{\mathbb{T}}\Vert_{W^{3,2}(\Oeta)}^2 
\big)
\dt.
\end{aligned}
\end{equation}
By combining \eqref{mR1}-\eqref{mR2x} with \eqref{d4} and absorbing $\epsilon$-terms into the left-hand side,  we conclude that
\begin{equation}
\begin{aligned}
\label{d5}
\sup_{I_*}
\big(
\Vert   \partial_t^2 \naby \eta\Vert_{L^2(\omega)}^2
&+
\Vert   \partial_t \naby\Dely \eta\Vert_{L^2(\omega)}^2
\big) 
 +\int_{I_*}\big(   \Vert\bu\Vert_{W^{4,2}(\Oeta)}^2+\Vert p\Vert_{W^{3,2}(\Oeta)}^2\big)\dt
\\&+
\int_{I_*}\big( \Vert\partial_t^2\bu\Vert_{L^{2}(\Oeta)}^2 + \Vert \partial_t\bu\Vert_{W^{2,2}(\Oeta)}^2+\Vert \partial_tp\Vert_{W^{1,2}(\Oeta)}^2\big)\dt
\\&+
\int_{I_*} 
\big(
\Vert \partial_t \dely^2 \eta\Vert_{L^2(\omega)}^2 
 +
\Vert  \partial_t^2 \dely \eta\Vert_{L^2(\omega)}^2 
+
\Vert  \partial_t^3 \eta\Vert_{L^2(\omega)}^2 
+
\Vert \Dely^3\eta\Vert_{L^{2}(\omega)}^2
\big)
\dt
\\
\leq& \mathcal{D}_2
\end{aligned}
\end{equation} 
where
\begin{align*}
\mathcal{D}_2:=&
c\Vert \eta_\star\Vert_{W^{3,2}(\omega)}^2
 +
 c\Vert \eta_0\Vert_{W^{5,2}(\omega)}^2
  +
c
 \Vert
  \bu_0\Vert_{W^{3,2}(\Omega_{\eta_0})}^2
   \\&
    +
 c \Vert \underline{\bT}(0)
  \Vert_{W^{2,2}(\Omega_{\eta_0})}^2
  +
   c\Vert \bff(0)\Vert_{W^{1,2}(\Omega_{\eta_0})}^2   
     +
  c\Vert g(0) \Vert_{W^{1,2}(\omega)}^2  
   \\&+
   c
   \int_{I_*} \big(\Vert \bff\Vert_{W^{2,2}(\Oeta)}^2
  +
  \Vert g \Vert_{W^{2,2}(\omega)}^2
  +
  \Vert\partial_tg\Vert_{L^2(\omega)}^2
  +
  \Vert  \underline{\bT}\Vert_{W^{3,2}(\Omega_\eta)}^2
      \big)
\dt
\\&
+
c
\int_{I_*}\big( \Vert\partial_t\bff\Vert_{L^{2}(\Oeta)}^2+\Vert \partial_t \underline{\bT}\Vert_{W^{1,2}(\Oeta)}^2\big)\dt.
\end{align*}
Finally, due to the continuity of the following embedding
\begin{align*}
L^2(I_*;W^{6,2}(\omega)) \cap W^{1,2}(I_*;W^{4,2}(\omega)) \hookrightarrow L^\infty(I_*;W^{5,2}(\omega)),
\end{align*}
we can also bound $\eta$ in the latter space in terms of $\mathcal{D}_2$. This finishes the proof.
\end{proof}

We now have all the required estimates for the fluid-structure subproblem so we proceed to the construction of a strong solution for the solute subproblem. Afterwards, we will use a fixed-point argument to close the full system.
\subsection{The solute subproblem}

For a known flexible domain $\Omega_\zeta$ and a known solenoidal vector field $\bv$, we aim in this section to construct a strong solution of the following subproblem:
\begin{align}  
\partial_t \bT + (\bv\cdot \nabx) \bT
=
(\nabx \bv)\bT + \bT(\nabx \bv)^\top+2\mathbb{D}(\nabx \bv) - 2 \bT\label{tEquAlone} 
\end{align}
on $I\times\Omega_\zeta\subset \mathbb R^{1+3}$ subject to the following initial condition
\begin{align} 
&\bT(0,\cdot)=\bT_0(\cdot) &\text{in }\Omega_{\zeta(0)}.
\label{initialCondSolvSubPro}  
\end{align} 
Let us start with a precise definition of what we mean by a strong solution.  
\begin{definition}
\label{def:strsolmartFP}
Assume that  $(\bT_0, \bv, \zeta)$ satisfies
\begin{equation}
\begin{aligned}
\label{fokkerPlanckDataAlone}
&\bT_0 \in  W^{3,2}( \Omega_{\zeta(0)}), 
\\&
\bv\in
W^{1,2} \big(I; W^{2,2}_{\divx}(\Omega_\zeta ) \big)\cap
L^2\big(I;W^{4,2}(\Omega_\zeta )  \big),
\\
&\zeta\in W^{3,2}\big(I;L^{2}(\omega)  \big)\cap L^{2}\big(I;W^{6,2}(\omega)  \big)
,
\\& \bv  \circ \bm{\varphi}_{\zeta} =(\partial_t\zeta)\bn
\quad \text{on }I \times \omega,  \quad\|\zeta\|_{L^\infty(I\times\omega)}<L.
\end{aligned}
\end{equation}
We call
$\bT$
a \textit{strong solution} of   \eqref{tEquAlone}-\eqref{initialCondSolvSubPro} with dataset $(\bT_0, \bv, \zeta)$ if 
\begin{itemize}
\item[(a)] $\bT \in   W^{1,\infty} \big(I; W^{2,2}(\Omega_\zeta ) \big)\cap L^\infty\big(I;W^{3,2}(\Omega_\zeta )  \big)$; 
\item[(b)] Equation \eqref{tEquAlone} is satisfied a.e. in $I\times\Omega_\zeta$.
\end{itemize}
\end{definition}
\noindent We now formulate our result on the existence of a  unique strong solution of  \eqref{tEquAlone}-\eqref{initialCondSolvSubPro}.
\begin{theorem}\label{thm:mainFP}
Let  $(\bT_0,  \bv, \zeta)$ satisfy  \eqref{fokkerPlanckDataAlone}.
Then there is a unique strong solution $\bT $ of  \eqref{tEquAlone}-\eqref{initialCondSolvSubPro}, in the sense of Definition \ref{def:strsolmartFP},
 such that
\begin{equation}
\begin{aligned} \label{thm:mainFPext}
\sup_{t\in I} \Big( \Vert \bT(t)\Vert_{W^{3,2}(\Ozeta)}^2 
&+
\Vert \partial_t\bT(t)\Vert_{W^{2,2}(\Ozeta)}^2 
\Big) 
\\&
\lesssim 
 \Big(
  \Vert  \bT_0\Vert_{W^{3,2}\Omega_{\zeta(0)})}^2 
   +
 \int_I
\Vert  \bv\Vert_{W^{4,2}(\Ozeta)}^2\dt
+
\int_I\Vert  \partial_t\bv
\Vert_{W^{2,2}(\Ozeta)}^2
\Big)  
\\&\qquad
\times
\Big(1+\int_I\Vert  \bv
\Vert_{W^{4,2}(\Ozeta)}^2
+
\int_I\Vert  \partial_t\bv
\Vert_{W^{2,2}(\Ozeta)}^2
\bigg)
  e^{c\int_I
\Vert  \bv\Vert_{W^{4,2}(\Ozeta)}  \dt} 
\end{aligned}
\end{equation} 
holds.
\end{theorem}

\begin{proof}
First of all, we note that given the regularity of the input $\zeta$, the embedding
\begin{align*}
W^{3,2}\big(I;L^{2}(\omega)  \big)\cap L^{2}\big(I;W^{6,2}(\omega)  \big)
\hookrightarrow C^{1,\beta}(I\times \omega)
\end{align*}
is continuous for all $\beta<\tfrac{1}{2}$, and so $\partial\Ozeta$ is of class $C^{1,\beta}$. The shell is thus, better than Lipschitz.
 
We wish to apply the method of characteristic to \eqref{tEquAlone}-\eqref{initialCondSolvSubPro}. Because the term $\bT(\nabx\bv)^\top$ is not in the right form and we are unable to rewrite it as $\mathbb{A}\bT$ for some matrix $\mathbb{A}\in \R^{3\times 3}$, we instead transform the matrix $\bT \in\R^{3\times 3}$ into a vector  $ \mathbf{T}\in\R^{9}$ by way of the mapping
\begin{align*}
\begin{pmatrix}
T_{11}&T_{12}&T_{13}\\
T_{21}&T_{22}&T_{23}\\
T_{31}&T_{32}&T_{33}
\end{pmatrix} 
\mapsto
(T_{11},T_{12},T_{13},T_{21},T_{22},T_{23},T_{31},T_{32}T_{33})^\top.
\end{align*}
In so doing, the mapping $\bT  \mapsto\mathbf{T} $ transforms \eqref{tEquAlone}-\eqref{initialCondSolvSubPro}   into the  vector-valued system
\begin{equation}
\begin{aligned}  
\label{vectorT}  
&\frac{\dd\mathbf{T} }{\dt}  
+ (\bv\cdot \nabx) \mathbf{T} 
=
(\mathbb{W}   - 2\mathbb{I}) \mathbf{T} +\mathbf{W} 
&\text{in }I\times\Omega_{\zeta}, 
\\
&\mathbf{T} (0,\bx)=\mathbf{T} _0(\bx) &\text{in }\Omega_{\zeta(0)}
\end{aligned} 
\end{equation}
which is in a  desired form.
Here,  we use the fact that the  tensor $[(\nabx \bv)\bT + \bT(\nabx \bv)^\top ]\in\R^{3\times3} $ transforms to a vector $\mathbb{W} \mathbf{T}  \in\R^9 $ 
where $\mathbb{W} \in\R^{9\times9}$ is the block matrix given by
\begin{align*}
\mathbb{W} =
\begin{pmatrix}
\mathbb{W}_{11}&\mathbb{W}_{12}&\mathbb{W}_{13}\\
\mathbb{W}_{21}&\mathbb{W}_{22}&\mathbb{W}_{23}\\
\mathbb{W}_{31}&\mathbb{W}_{32}&\mathbb{W}_{33}
\end{pmatrix}_{9\times9}
\end{align*}
with sub entries
\begin{align*}
&\mathbb{W}_{11}
=
\begin{pmatrix}
2\partial_1v_1&\partial_1v_2&\partial_1v_3\\
\partial_2v_1&\partial_1v_1+\partial_2v_2&\partial_2v_3\\
\partial_3v_1&\partial_3v_2&\partial_1v_1+\partial_3v_3
\end{pmatrix},
&\mathbb{W}_{12}
=
\begin{pmatrix}
\partial_1v_2&0&0\\
0&\partial_1v_2&0\\
0&0&\partial_1v_2
\end{pmatrix},
\\
&\mathbb{W}_{13}
=
\begin{pmatrix}
\partial_1v_3&0&0\\
0&\partial_1v_3&0\\
0&0&\partial_1v_3
\end{pmatrix},
&\mathbb{W}_{21}
=
\begin{pmatrix}
\partial_2v_1&0&0\\
0&\partial_2v_1&0\\
0&0&\partial_2v_1
\end{pmatrix},
\\
&\mathbb{W}_{22}
=
\begin{pmatrix}
\partial_1v_1+\partial_2v_2&\partial_1v_2&\partial_1v_3\\
\partial_2v_1&2\partial_2v_2&\partial_2v_3\\
\partial_3v_1&\partial_3v_2&\partial_2v_2+\partial_3v_3
\end{pmatrix},
&\mathbb{W}_{23}
=
\begin{pmatrix}
\partial_2v_3&0&0\\
0&\partial_2v_3&0\\
0&0&\partial_2v_3
\end{pmatrix},
\\
&\mathbb{W}_{31}
=
\begin{pmatrix}
\partial_3v_1&0&0\\
0&\partial_3v_1&0\\
0&0&\partial_3v_1
\end{pmatrix},
&\mathbb{W}_{32}
=
\begin{pmatrix}
\partial_3v_2&0&0\\
0&\partial_3v_2&0\\
0&0&\partial_3v_2
\end{pmatrix},
\\&\mathbb{W}_{33}
=
\begin{pmatrix}
\partial_1v_1+\partial_3v_3&\partial_1v_2&\partial_1v_3\\
\partial_2v_1&\partial_2v_2+\partial_3v_3&\partial_2v_3\\
\partial_3v_1&\partial_3v_2&2\partial_3v_3
\end{pmatrix}.
\end{align*}
Also, the  symmetric  tensor $2\mathbb{D}(\nabx \bv) \in\R^{3\times3}$ transforms to a vector $\mathbf{W} \in\R^9 $
given by
\begin{align*}
\mathbf{W} 
=(
2\partial_1v_1,
\partial_1v_2+
\partial_2v_1,
\partial_1v_3+
\partial_3v_1,
\partial_1v_2+
\partial_2v_1,
2\partial_2v_2,
\partial_2v_3+
\partial_3v_2,
\partial_1v_3+
\partial_3v_1,
\partial_2v_3+
\partial_3v_2,
2\partial_3v_3)^\top.
\end{align*}
With these definitions, we directly obtain the following equivalent norm relations
\begin{align}
&\Vert\mathbf{T}\Vert_{W^{s,p}(\Ozeta)}\sim \Vert\bT\Vert_{W^{s,p}(\Ozeta)},\label{equivNorm0}
\\
&\Vert\mathbf{W}\Vert_{W^{s,p}(\Ozeta)}\sim \Vert\mathbb{W}\Vert_{W^{s,p}(\Ozeta)}\sim \Vert\bv\Vert_{W^{s+1,p}(\Ozeta)}.\label{equivNorm}
\end{align}
for any $s\geq0$ and $p\geq1$.
We now apply the method of characteristics to \eqref{vectorT}. Here, we note that
due to the kinematically compatible  condition
$\bv  \circ \bm{\varphi}_{\zeta} =(\partial_t\zeta)\bn$  at the boundary, a characteristic curve stays in the moving domain for all times.
%
%
%
%
%
Thus, we define $\bm{\Phi}(t,\bx)$ to be the unique solution (follows from Picard--Lindel\"of) to the equation:
\begin{align}
&\frac{\dd \bm{\Phi}(t,\bx)}{\dt }=\bv(t ,\bm{\Phi}(t,\bx)),
\\
&\bm{\Phi}(0,\bx)= \bx
\end{align}
so that $\mathbf{T} $ solves
\begin{equation}
\begin{aligned}  
\label{vectorT1}  
&\frac{\dd\mathbf{T} }{\dt}   
=
\big(\mathbb{W}(t ,\bm{\Phi}(t,\bx))   - 2\mathbb{I}\big) \mathbf{T}(t ,\bm{\Phi}(t,\bx)) +\mathbf{W}(t ,\bm{\Phi}(t,\bx)) 
&\text{in }I\times\Omega_{\zeta} ,
\\
&\mathbf{T} (0 ,\bm{\Phi}(0,\bx))=\mathbf{T} _0(\bx) &\text{in }\Omega_{\zeta(0)}. 
\end{aligned} 
\end{equation}
More precisely, the solution to \eqref{vectorT1}  is obtained by employing the integrating factor
\begin{align*}
e^{\int_0^t\big(2\mathbb{I}-\mathbb{W}(t' ,\bm{\Phi}(t',\bx))   \big) \dt'}
\end{align*}
leading to
\begin{align}
\label{vectorT2}  
\mathbf{T}(t ,\bx) = \frac{ 
\mathbf{T}_0(\bx )
+
\int_0^t
e^{ \int_0^{t'}(2\mathbb{I}-\mathbb{W}(s ,\bm{\Phi}(s,\bx)) )\dd s} 
\mathbf{W}(t' ,\bm{\Phi}(t',\bx))
\dt'}{e^{ \int_0^t (2\mathbb{I}-\mathbb{W}(t' ,\bm{\Phi}(t',\bx)) )\dd t'}}.
\end{align}
Note that uniqueness of the solution \eqref{vectorT2}  is a direct consequence of the linearity of \eqref{vectorT1}  since the difference of any two solution will satisfy the easier homogeneous version of \eqref{vectorT1} with
\begin{align*}
\mathbf{W}(t ,\bm{\Phi}(t,\bx)) \equiv0, \qquad \mathbf{T}_0(\bx)=0.
\end{align*}
Now observe that if $\bv   \in L^1(I;W^{4,2}(\Ozeta))$, then by
\eqref{equivNorm},  $ \mathbb{W}   \in L^1(I;W^{3,2}(\Ozeta))$. Thus, we can use Jensen's inequality and convexity of the exponential function to conclude that
\begin{align*}
\Vert e^{ \pm\int_0^t  (2\mathbb{I}-\mathbb{W}(t' ,\bm{\Phi}(t',\bx)) )\dd t'} \Vert_{W^{3,2}(\Ozeta)} \lesssim 1
\end{align*}
for $\bv   \in L^1(I;W^{4,2}(\Ozeta))$.
Consequently,  since $W^{3,2}(\Ozeta)$ is a Banach algebra, we obtain from \eqref{vectorT2} that
\begin{align}
\label{vectorT3}  
\mathbf{T}\in C(\overline{I};W^{3,2}(\Ozeta))
\end{align}
provided $\mathbf{T}_0\in W^{3,2}(\Omega_{\zeta_0})$ and from \eqref{vectorT3} and \eqref{equivNorm0}, we finally obtain the existence of the unique strong solution
\begin{align}
\label{vectorT5}  
\bT\in C(\overline{I};W^{3,2}(\Ozeta)).
\end{align}
 For the estimate \eqref{thm:mainFPext}, we consider all multi-indexes satisfying $|\alpha|\leq 3$ and apply $\partial_\bx^\alpha$ to \eqref{tEquAlone}. This results in
\begin{equation}
\begin{aligned}
\label{partialAlphaT}
\partial_t \partial_\bx^\alpha\bT + (\bv\cdot \nabx) \partial_\bx^\alpha\bT
=&
(\nabx \bv)\partial_\bx^\alpha\bT + (\partial_\bx^\alpha\bT)(\nabx \bv)^\top+2\partial_\bx^\alpha\mathbb{D}(\nabx \bv) \\&- 2 \partial_\bx^\alpha\bT+R_1+R_2+R_3 
\end{aligned}
\end{equation}
where
\begin{align*}
R_1
    &=
(\bv\cdot \nabx)\partial_\bx^\alpha \bT
-
\partial_\bx^\alpha((\bv\cdot \nabx) \bT),
\\
R_2
    &=
 (\partial_\bx^\alpha\bT)(\nabx \bv)^\top
-
\partial_\bx^\alpha(\bT(\nabx \bv)^\top ),
\\
R_3
    &=
(\nabx \bv) (\partial_\bx^\alpha\bT) 
-
\partial_\bx^\alpha((\nabx \bv)\bT  )
\end{align*}
are commutators which can be estimated using the  commmutator estimate  \cite[(2.3)]{glatt2014local} to obtain
\begin{align*}
\Vert R_1\Vert_{L^2(\Ozeta)}^2
&\lesssim
\Vert  \bT\Vert_{W^{3,2}(\Ozeta)}\Vert  \bv\Vert_{W^{3,2}(\Ozeta)},
\\
\Vert R_2\Vert_{L^2(\Ozeta)}^2+
\Vert R_3\Vert_{L^2(\Ozeta)}^2
&\lesssim
\Vert  \bT\Vert_{W^{2,2}(\Ozeta)}\Vert  \bv\Vert_{W^{4,2}(\Ozeta)}.
\end{align*}
Thus, if we now test \eqref{partialAlphaT} with $\partial_\bx^\alpha\bT$ and use Reynold's transport theorem, we obtain
\begin{align*}
 \Vert \partial_\bx^\alpha\bT(t)\Vert_{L^2(\Ozeta)}^2 
 +\int_0^t\Vert \partial_\bx^\alpha\bT\Vert_{L^2(\Ozeta)}^2 \dt'
&\lesssim
 \Vert \partial_\bx^\alpha\bT_0\Vert_{L^2(\Omega_{\zeta(0)})}^2 
 +
 \int_I
\Vert  \bv\Vert_{W^{4,2}(\Ozeta)}^2\dt
\\&\quad
 +
 \int_I
\Vert  \bv\Vert_{W^{4,2}(\Ozeta)}  \Vert  \bT\Vert_{W^{3,2}(\Ozeta)}^2\dt.
\end{align*}
If we sum the estimate above over all $|\alpha|\leq3$, we obtain by Gr\"onwall's lemma that
\begin{align*}
  \sup_I \Vert  \bT(t)\Vert_{W^{3,2}(\Ozeta)}^2 
  \lesssim
  \Big(
  \Vert  \bT_0\Vert_{W^{3,2}\Omega_{\zeta(0)})}^2 
   +
 \int_I
\Vert  \bv\Vert_{W^{4,2}(\Ozeta)}^2\dt
\Big)
  e^{c\int_I
\Vert  \bv\Vert_{W^{4,2}(\Ozeta)}  \dt}.
\end{align*}
By returning to the original equation  \eqref{tEquAlone}, it follows that  
  \begin{align*}
\sup_I\Vert\partial_t\bT\Vert_{W^{2,2}(\Ozeta)}^2
     &\lesssim 
\sup_I\Vert \bT \Vert_{W^{3,2}(\Ozeta)}^2
\Big(1+ \sup_I\Vert  \bv
\Vert_{W^{3,2}(\Ozeta)}^2 
\Big)
+\sup_I\Vert  \bv 
\Vert_{W^{3,2}(\Ozeta)}^2 
\\&
\lesssim 
 \Big(
  \Vert  \bT_0\Vert_{W^{3,2}\Omega_{\zeta(0)})}^2 
   +
 \int_I
\Vert  \bv\Vert_{W^{4,2}(\Ozeta)}^2\dt
\Big)  
\\&\quad\times
\Big(1+
 \sup_I\Vert  \bv
\Vert_{W^{3,2}(\Ozeta)}^2 
\Big)
  e^{c\int_I
\Vert  \bv\Vert_{W^{4,2}(\Ozeta)}  \dt}
+
\sup_I\Vert  \bv 
\Vert_{W^{3,2}(\Ozeta)}^2
\\&
\lesssim 
 \Big(
  \Vert  \bT_0\Vert_{W^{3,2}\Omega_{\zeta(0)})}^2 
   +
 \int_I
\Vert  \bv\Vert_{W^{4,2}(\Ozeta)}^2\dt
\Big)  
\\&\quad\times
\Big(1+\int_I\Vert  \bv
\Vert_{W^{4,2}(\Ozeta)}^2
+
\int_I\Vert  \partial_t\bv
\Vert_{W^{2,2}(\Ozeta)}^2
\bigg)
  e^{c\int_I
\Vert  \bv\Vert_{W^{4,2}(\Ozeta)}  \dt}
\\&
\qquad+
 \int_I\Vert  \bv
\Vert_{W^{4,2}(\Ozeta)}^2
+
\int_I\Vert  \partial_t\bv
\Vert_{W^{2,2}(\Ozeta)}^2
 \end{align*}
where we have used the continuous embedding
\begin{align*}
    L^2(I;W^{4,2}(\Ozeta))\cap 
    W^{1,2}(I;W^{2,2}(\Ozeta))
    \hookrightarrow
    W^{2/3,2}(I;W^{7/2,2}(\Ozeta))
    \hookrightarrow
    L^{\infty}(I;W^{3,2}(\Ozeta))
\end{align*}
in the last step. Combining the two supremum estimates above yield 
\begin{align*} 
\sup_{t\in I} \Big( \Vert \bT(t)\Vert_{W^{3,2}(\Ozeta)}^2 
+
\Vert \partial_t\bT(t)\Vert_{W^{2,2}(\Ozeta)}^2 
\Big)
\lesssim& 
 \int_I\Vert  \bv
\Vert_{W^{4,2}(\Ozeta)}^2
+
\int_I\Vert  \partial_t\bv
\Vert_{W^{2,2}(\Ozeta)}^2
\\&+ 
 \Big(
  \Vert  \bT_0\Vert_{W^{3,2}\Omega_{\zeta(0)})}^2 
   +
 \int_I
\Vert  \bv\Vert_{W^{4,2}(\Ozeta)}^2\dt
\Big) 
\\&\quad\times 
\Big(1+\int_I\Vert  \bv
\Vert_{W^{4,2}(\Ozeta)}^2
+
\int_I\Vert  \partial_t\bv
\Vert_{W^{2,2}(\Ozeta)}^2
\bigg)
\\&\quad\times
  e^{c\int_I
\Vert  \bv\Vert_{W^{4,2}(\Ozeta)}  \dt} 
\end{align*} 
%
which can further be bounded by the desired simplified right-hand side given in
\eqref{thm:mainFPext}. This completes the proof.
\end{proof}

\section{The fully coupled system}
\label{sec:fully}
To construct a local-in-time solution for the fully coupled system, we resort to a fixed-point argument. The existence of such a fixed-point will follow from closedness and contraction properties for a suitable map. More precisely, to show the closedness property, we consider the mapping $\mathtt{T}=\mathtt{T}_1\circ\mathtt{T}_2$ with
\begin{align*}
\mathtt{T}(\underline{\bT})= \bT , \qquad \mathtt{T}_2(\underline{\bT} )=(\eta,\bu,p), \qquad \mathtt{T}_1(\eta,\bu,p)= \bT 
\end{align*}
defined on the space
\begin{align*}
X_\eta&:=L^\infty (I_*;W^{3,2}(\Oeta)  )
\cap W^{1,\infty} (I_*;W^{2,2}(\Oeta)  ), 
\end{align*}
equipped with its canonical norm $\Vert \cdot\Vert_{X_\eta}$ and where $I_*$ is the local time yet to be determined. 
We now consider the subset 
\begin{align*}
B_R:=\big\{  \underline{\bT} \in X_\eta  \,:\,\Vert \underline{\bT} \Vert_{X_\eta  }^2 \leq R^2  \big\}.
\end{align*}
and show that $\mathtt{T}:X_\eta \rightarrow X_\eta$ maps $B_R$ into $B_R$, i.e., for any $ \underline{\bT} \in   B_R$, we have that 
\begin{align}
\label{balltoball}
\Vert  \bT \Vert_{X_\eta }^2=\Vert \mathtt{T}(\underline{\bT} )\Vert_{X_\eta }^2=\Vert \mathtt{T}_1\circ \mathtt{T}_2(\underline{\bT} )\Vert_{X_\eta }^2=\Vert \mathtt{T}_1(\eta,\bu, p) \Vert_{X_\eta }^2 \leq R^2.
\end{align}
Indeed, for 
\begin{align*}
(\eta,\bu)\in&  
W^{1,\infty}\big(I_*;W^{3,2}(\omega)  \big)\cap L^{\infty}\big(I_*;W^{5,2}(\omega)  \big) \cap  W^{3,2}\big(I_*; L^{2}(\omega)  \big)\cap  W^{2,2}\big(I_*;W^{2,2}(\omega)  \big) \cap  L^2\big(I_*;W^{6,2}(\omega)  \big)
\\&\times
W^{2,2} \big(I_*; L^{2}_{\divx}(\Oeta ) \big)\cap W^{1,2} \big(I_*; W^{2,2}(\Oeta ) \big)\cap L^2\big(I_*;W^{4,2}(\Oeta)  \big), 
\end{align*}
let $\bT$ be the unique strong solution of \eqref{tEquAlone}-\eqref{initialCondSolvSubPro} with dataset $( \bT_0, \bu, \eta)$ as shown in Theorem \ref{thm:mainFP}. 
As shown in Theorem \ref{thm:mainFP}, the tensor $\bT$ will satisfy 
\begin{equation}
\begin{aligned} 
\label{tXeta}
\Vert  \bT \Vert_{X_\eta }^2&
\lesssim 
  \Big(
  \Vert  \bT_0\Vert_{W^{3,2}\Omega_{\eta_0})}^2 
   +
\int_{I_*}\Vert  \bu
\Vert_{W^{4,2}(\Oeta)}^2\dt
+
\int_{I_*}\Vert \partial_t \bu
\Vert_{W^{2,2}(\Oeta)}^2\dt
\Big)
\\&\qquad
\times
  \Big(
 1  +
\int_{I_*}\Vert  \bu
\Vert_{W^{4,2}(\Oeta)}^2\dt
+
\int_{I_*}\Vert \partial_t \bu
\Vert_{W^{2,2}(\Oeta)}^2\dt
\Big)  e^{c\int_{I_*}
\Vert  \bu\Vert_{W^{4,2}(\Oeta)}  \dt}.
\end{aligned}
\end{equation}    
On the other hand, for $ \underline{\bT}\in B_R$,   let $( \eta, \bu,  p )$ be the unique strong solution  of  \eqref{divfreeAlone}--\eqref{interfaceAlone} with data $(\bff, g, \eta_0, \eta_\star, \bu_0, \underline{\bT})$ as shown in  Theorem \ref{thm:BM}. Then by Proposition \ref{prop:BM}, we can control 
\begin{equation}
\begin{aligned}
\label{twoVeloOneShell}  
\int_{I_*}\Vert  \bu
\Vert_{W^{4,2}(\Oeta)}^2\dt
+
\int_{I_*}\Vert \partial_t \bu
\Vert_{W^{2,2}(\Oeta)}^2\dt
\end{aligned}
\end{equation}
in terms of
\begin{equation}
\begin{aligned}\nonumber
\mathcal{E}(\mathrm{data})
+
\int_{I_*} 
\Big( 
\Vert \underline{\bT}
\Vert_{W^{3,2}(\Oeta)}^2
+
\Vert  \partial_t\underline{\bT}
\Vert_{W^{1,2}(\Oeta)}^2
\Big)
\dt,
\end{aligned}
\end{equation} 
where
\begin{equation}
\begin{aligned}\nonumber
\mathcal{E}(\mathrm{data})
&:=
\Vert \eta_\star\Vert_{W^{3,2}(\omega)}^2
 +
 \Vert \eta_0\Vert_{W^{5,2}(\omega)}^2
  +
 \Vert
  \bu_0\Vert_{W^{3,2}(\Omega_{\eta_0})}^2
   \\&
    +
  \Vert \underline{\bT}(0)
  \Vert_{W^{2,2}(\Omega_{\eta_0})}^2
  +
   \Vert \bff(0)\Vert_{W^{1,2}(\Omega_{\eta_0})}^2    
   \\&+
   \int_{I_*} \big(\Vert g \Vert_{W^{2,2}(\omega)}^2
  +
\Vert \partial_t g \Vert_{L^{2}(\omega)}^2
    \big)
\dt
   \\&+
   \int_{I_*} \big( \Vert \bff\Vert_{W^{2,2}(\Oeta)}^2
  +
 \Vert \partial_t\bff\Vert_{L^2(\Oeta)}^2
    \big)
\dt.
\end{aligned}
\end{equation}  
However, by using H\"older's inequality in time, we obtain 
\begin{equation}
\begin{aligned}
\label{aprioriFP1}
\int_{I_*} 
\Big( 
\Vert \underline{\bT}
\Vert_{W^{3,2}(\Oeta)}^2
+
\Vert  \partial_t\underline{\bT}
\Vert_{W^{1,2}(\Oeta)}^2
\Big)
\dt
\lesssim
T_*
\Vert \underline{\bT}
\Vert_{X_\eta}^2.
\end{aligned}
\end{equation}
We also note that for $R>0$ large enough and $T_*$ very small (say $T_*<\tfrac{1}{R^2}$), we can use H\"older inequality in time and \eqref{twoVeloOneShell}-\eqref{aprioriFP1}  to obtain
\begin{equation}
\begin{aligned}\label{manyE}
 e^{c\int_{I_*}
\Vert  \bu\Vert_{W^{4,2}(\Ozeta)}  \dt}
&\lesssim
 e^{cT_*^{1/2}\sqrt{\int_{I_*}
\Vert  \bu\Vert_{W^{4,2}(\Ozeta)}^2  \dt}}
\\
&\lesssim
e^{ cT_*^{1/2}\sqrt{ \mathcal{E}(\mathrm{data})
  +
  T_*
\Vert \underline{\bT}
\Vert_{X_\eta}^2}}
\\
&\lesssim
e^{ cT_*^{1/2}\sqrt{c+T_*R^2}}.
\end{aligned}
\end{equation}
If we now combine all of \eqref{tXeta}-\eqref{manyE}, then for $R>0$ large enough and $T_*$ very small (say $T_*<\tfrac{1}{R^2}$),
we can conclude that
 
\begin{align*}
\Vert  \bT \Vert_{X_\eta }^2& 
\lesssim
  \Big(
 1
  +
  T_*
R^2
\Big)  
  \Big(
 c +
  T_*
R^2
\Big)  e^{ cT_*^{1/2}\sqrt{c+T_*R^2}}
\\
& 
\lesssim
  \Big(
 1
  +
1
\Big)  
  \Big(
 c +
1
\Big)  e^{ c R^{-1}\sqrt{c+1}}
\\
& 
\leq c_1 
   e^{ c_2R^{-1}}
\\
& 
\leq R^2 
\end{align*} 
for $R\ln( {R^2}/{c_1})\geq c_2$. Thus,  $\mathtt{T}:X_\eta \rightarrow X_\eta$ maps $B_R$ into $B_R$ as desired in \eqref{balltoball}.
 
Now, for the contraction property, we consider the superset $Y_\eta \supseteq X_\eta$ given by 
\begin{align*} 
Y_\eta&:=  L^{\infty} (I_*;L^2(\Oeta ) ),
\end{align*}
and equipped with its canonical norm $\Vert \cdot\Vert_{Y_\eta}$, and show that
\begin{equation}
\begin{aligned}
\label{contrEst00}
\Vert \mathtt{T}( \underline{\bT}_1-\underline{\bT}_2) \Vert_{Y_{\eta_1} }^2
&\leq
\tfrac{1}{2}
\Vert (   \underline{\bT}_1-\underline{\bT}_2)\Vert_{Y_{\eta_1} }^2
\end{aligned}
\end{equation} 
hold for any pair of strong solutions $  \underline{\bT}_i\in X_{\eta_1} $, $i=1,2$ for the solute subproblem \eqref{tEquAlone}-\eqref{initialCondSolvSubPro}
with dataset $(\bT_0, \bu_i, \eta_i)$, $i=1,2$, respectively. To show \eqref{contrEst00}, we note that since the fluid domain depends on the deformation of the shell, we have to transform one solution, say $\bT_2$, to the domain of $\bT_1$ to get a difference estimate that is well-defined on a single domain. For this, we set  $\overline{\bT}_2=\bT_2\circ \bm{\Psi}_{\eta_2-\eta_1}$, $\overline{\bu}_2=\bu_2\circ \bm{\Psi}_{\eta_2-\eta_1}$  and obtain (see \cite{mensah2023weak} for further details)
 \begin{align*}
\partial_t\overline{\bT}_2
+
\overline{\bu}_2 \cdot\nabx \overline{\bT}_2 
&=
\nabx\overline{\bu}_2\overline{\bT}_2
+
\overline{\bT}_2 (\nabx\overline{\bu}_2)^\top
+
2\mathbb{D}(\nabx\overline{\bu}_2)
-
2 \overline{\bT}_2  
+
\mathbb{H}_{\eta_2-\eta_1}( \overline{\bu}_2,\overline{\bT}_2)
\end{align*}
defined on ${I_*} \times \Omega_{\eta_1}$ where
\begin{align*}
\mathbb{H}_{\eta_2-\eta_1}( \overline{\bu}_2,\overline{\bT}_2)
=&
(1-J_{\eta_2-\eta_1})\partial_t\overline{\bT}_2
-
J_{\eta_2-\eta_1} \nabx \overline{\bT}_2\cdot\partial_t\bm{\Psi}_{\eta_2-\eta_1}^{-1}\circ \bm{\Psi}_{\eta_2-\eta_1}
\\&+
\nabx\overline{\bu}_2(\mathbb{B}_{\eta_2-\eta_1}-\mathbb{I})(\overline{\bT}_2+\mathbb{I})
+
(\overline{\bT}_2+\mathbb{I})(\mathbb{B}_{\eta_2-\eta_1}-\mathbb{I})^\top (\nabx\overline{\bu}_2)^\top
\\&+
\overline{\bu}_2 \cdot\nabx \overline{\bT}_2 (\mathbb{I}-\mathbb{B}_{\eta_2-\eta_1})
+
2
(1-J_{\eta_2-\eta_1})  \overline{\bT}_2.
\end{align*}
We now consider the difference
\begin{align*}
\bT_{12}:=\bT_1-\overline{\bT}_2,
\quad
\bu_{12}=\bu_1-\overline{\bu}_2,  
\quad
\eta_{12}=\eta_1-\eta_2
\end{align*}
and find that $\bT_{12}$ solves
\begin{equation}
\begin{aligned}
\label{diffEquaSigma}
\partial_t \bT_{12}+ \mathbf{u}_1\cdot \nabx \bT_{12}
&=
\nabx \bu_1\bT_{12} + \bT_{12}(\nabx \bu_1)^\top 
+
2\mathbb{D}(\nabx \bu_{12})
- 2\bT_{12}   
\\&
+
\nabx \bu_{12}\overline{\bT}_2 + \overline{\bT}_2(\nabx \bu_{12})^\top 
- \mathbf{u}_{12}\cdot \nabx \overline{\bT}_2 
-
\mathbb{H}_{-\eta_{12}}( \overline{\bu}_2,\overline{\bT}_2)
\end{aligned}
\end{equation}
on ${I_*} \times \Omega_{\eta_1}$ with identically zero initial condition.
If we now test \eqref{diffEquaSigma} with $\bT_{12}$, then for $t\in {I_*}$, we obtain
\begin{equation}
\begin{aligned}
\frac{1}{2}&
\frac{\dd}{\dt}\Vert \bT_{12}\Vert_{L^2(\Omega_{\eta_1})}^2 
+ 
2
\Vert \bT_{12}\Vert_{L^2(\Omega_{\eta_1})}^2
\\&
\leq 
2\int_{\Omega_{\eta_1}}\vert\nabx\bu_1\vert
\vert \bT_{12}\vert^2\dx
+
2\int_{\Omega_{\eta_1}}\vert\nabx\bu_{12}\vert
\vert \bT_{12}\vert \dx
\\&
+
\int_{\Omega_{\eta_1}}
\big[
\nabx \bu_{12}\overline{\bT}_2 + \overline{\bT}_2(\nabx \bu_{12})^\top 
\big]:\bT_{12}\dx
-
\int_{\Omega_{\eta_1}}
\mathbf{u}_{12}\cdot \nabx \overline{\bT}_2
 :\bT_{12} \dx
\\& 
-
\int_{\Omega_{\eta_1}}
\mathbb{H}_{-\eta_{12}}( \overline{\bu}_2,\overline{\bT}_2) :\bT_{12} \dx
\\&
=:  K_1+\ldots+K_5
\end{aligned}
\end{equation} 
Firstly, by using the high regularity of the extra stress tensor and the small acceleration estimate \eqref{est:reg1},
\begin{align*}
\int_0^t K_1\dt'&\lesssim   
\int_0^t \Vert \bu_1\Vert_{W^{3,2}(\Omega_{\eta_1})}\Vert \bT_{12}\Vert_{L^2(\Omega_{\eta_1})}^2
\dt'.
 \end{align*}
Next, we note that
\begin{equation}
\begin{aligned}
\int_0^t
K_2\dt
&\lesssim
\int_0^t
\Vert\nabx\bu_{12}\Vert_{L^2(\Omega_{\eta_1})} 
\Vert \bT_{12}\Vert_{L^2(\Omega_{\eta_1})}\dt'  
\\
&\leq
c(\delta)
\int_0^t
\Vert \nabx\bu_{12}\Vert_{L^2(\Omega_{\eta_1})}^2\dt'
+
\delta
\int_0^t
\Vert  \bT_{12}\Vert_{L^2(\Omega_{\eta_1})}^2\dt' 
\end{aligned}
\end{equation} 
hold for any $\delta>0$ and similarly,
\begin{equation}
\begin{aligned}
\int_0^t
(K_3+K_4)\dt
&\lesssim
\int_0^t
\Vert\nabx\bu_{12}\Vert_{L^2(\Omega_{\eta_1})} 
\Vert \overline{\bT}_{2}\Vert_{L^\infty(\Omega_{\eta_1})}
\Vert \bT_{12}\Vert_{L^2(\Omega_{\eta_1})}\dt'  
\\&\quad+
\int_0^t
\Vert \bu_{12}\Vert_{L^4(\Omega_{\eta_1})} 
\Vert \nabx\overline{\bT}_{2}\Vert_{L^4(\Omega_{\eta_1})}
\Vert \bT_{12}\Vert_{L^2(\Omega_{\eta_1})}\dt'  
\\
&\leq
c(\delta)
\int_0^t
\Vert \nabx\bu_{12}\Vert_{L^2(\Omega_{\eta_1})}^2\Vert \overline{\bT}_{2}\Vert_{W^{2,2}(\Omega_{\eta_1})}^2\dt'
+
\delta
\int_0^t
\Vert  \bT_{12}\Vert_{L^2(\Omega_{\eta_1})}^2\dt' 
\\
&\leq
c(\delta)
\int_0^t
\Vert \nabx\bu_{12}\Vert_{L^2(\Omega_{\eta_1})}^2 \dt'
+
\delta
\int_0^t
\Vert  \bT_{12}\Vert_{L^2(\Omega_{\eta_1})}^2\dt' 
\end{aligned}
\end{equation} 
for any $\delta>0$. To deal with $K_5$ we first expand it as follows
\begin{equation}
\begin{aligned}
K_5 
=&
\int_{\Omega_{\eta_1}}
(1-J_{-\eta_{12}})\partial_t\overline{\bT}_2
 :\bT_{12} \dx
  \\&+2
\int_{\Omega_{\eta_1}}
(1-J_{-\eta_{12}}) \overline{\bT}_2  :\bT_{12} \dx
\\&-
\int_{\Omega_{\eta_1}}
J_{-\eta_{12}} \nabx \overline{\bT}_2\cdot\partial_t\bm{\Psi}_{-\eta_{12}}^{-1}\circ \bm{\Psi}_{-\eta_{12}}
 :\bT_{12} \dx
\\&+
\int_{\Omega_{\eta_1}}
\nabx\overline{\bu}_2(\mathbb{B}_{-\eta_{12}}-\mathbb{I})(\overline{\bT}_2+\mathbb{I})
 :\bT_{12} \dx
 \\&+
\int_{\Omega_{\eta_1}}
(\overline{\bT}_2+\mathbb{I})(\mathbb{B}_{-\eta_{12}}-\mathbb{I})^\top (\nabx\overline{\bu}_2)^\top
:\bT_{12} \dx
 \\&+
\int_{\Omega_{\eta_1}}
\overline{\bu}_2 \cdot\nabx \overline{\bT}_2 (\mathbb{I}-\mathbb{B}_{-\eta_{12}})
 :\bT_{12} \dx
\\=:&K_5^1+\ldots+K_5^6.
\end{aligned}
\end{equation} 
Then we have by interpolation 
\begin{align*}
K_5^1+K_5^2&\lesssim
\Vert  \eta_{12}\Vert_{W^{1,4}(\omega )}
\big(\Vert \partial_t\overline{\bT}_2\Vert_{L^{4}(\Omega_{\eta_1})}
+
\Vert  \overline{\bT}_2\Vert_{L^{4}(\Omega_{\eta_1})}
\big)
\Vert  \bT_{12}\Vert_{L^2(\Omega_{\eta_1})} 
\\
&\leq
\delta 
\Vert  \bT_{12}\Vert_{L^{2}(\Omega_{\eta_1})}^2
+
c(\delta )
\Vert  \eta_{12}\Vert_{W^{2,2}(\omega )}^2
\big(\Vert \partial_t\overline{\bT}_2\Vert_{W^{1,2}(\Omega_{\eta_1})}
+
\Vert  \overline{\bT}_2\Vert_{W^{1,2}(\Omega_{\eta_1})}
\big).
\end{align*} 
 Next, we have
\begin{align*}
K_5^3
&\lesssim
\Vert  \partial_t\eta_{12}\Vert_{W^{1,2}(\omega )}
\Vert \overline{\bT}_2\Vert_{W^{1,\infty}(\Omega_{\eta_1})}
\Vert \bT_{12}\Vert_{L^2(\Omega_{\eta_1})} 
\\
&\leq
\delta 
\Vert \bT_{12}\Vert_{L^2(\Omega_{\eta_1})}^2
+
c(\delta)
\Vert  \partial_t\eta_{12}\Vert_{W^{1,2}(\omega )}^2
\Vert \overline{\bT}_2\Vert_{W^{3,2}(\Omega_{\eta_1})}^2 
\end{align*}
as well as
\begin{align*}
K_5^4 +K_5^5 
&\lesssim  
\Vert \nabx\overline{\bu}_2\Vert_{L^{4}(\Omega_{\eta_1})}
\Vert \eta_{12}\Vert_{W^{1,4}(\omega )}
(\Vert \overline{\bT}_2\Vert_{L^{\infty}(\Omega_{\eta_1})}+1)
\Vert \bT_{12}\Vert_{L^2(\Omega_{\eta_1})} 
\\&
\leq
\delta
\Vert \bT_{12}\Vert_{L^2(\Omega_{\eta_1})}^2
+
c(\delta)
\Vert \eta_{12}\Vert_{W^{2,2}(\omega )}^2
(\Vert \overline{\bT}_2\Vert_{W^{2,2}(\Omega_{\eta_1})}^2
+1).
\end{align*}
Similarly,
\begin{align*}
K_5^6
&\lesssim  
\Vert  \overline{\bu}_2\Vert_{L^{\infty}(\Omega_{\eta_1})}
\Vert \nabx\overline{\bT}_2\Vert_{L^{4}(\Omega_{\eta_1})} 
\Vert \eta_{12}\Vert_{W^{1,4}(\omega )}
\Vert \bT_{12}\Vert_{L^2(\Omega_{\eta_1})} 
\\&
\leq
\delta
\Vert \bT_{12}\Vert_{L^2(\Omega_{\eta_1})}^2
+
c(\delta)
\Vert \eta_{12}\Vert_{W^{2,2}(\omega )}^2
\Vert \overline{\bT}_2\Vert_{W^{2,2}(\Omega_{\eta_1})}^2 .
\end{align*}
Thus,  
\begin{equation}
\begin{aligned}
\int_0^t 
K_5\dt'
\leq& 
\delta
\int_0^t 
\Vert  \bT_{12}\Vert_{L^{2}(\Omega_{\eta_1})}^2\dt'
 +
c
\int_0^t 
\Vert  \partial_{t'}\eta_{12}\Vert_{W^{1,2}(\omega )}^2\Vert  \overline{\bT}_2\Vert_{W^{3,2}(\Omega_{\eta_1})}^2\dt'
\\&+
c\int_0^t 
\Vert  \eta_{12}\Vert_{W^{2,2}(\omega )}^2\big(\Vert \partial_{t'}\overline{\bT}_2\Vert_{W^{1,2}(\Omega_{\eta_1})}^2
+
\Vert  \overline{\bT}_2\Vert_{W^{2,2}(\Omega_{\eta_1})}^2+1
\big)\dt'
\\\leq& 
\delta
\int_0^t 
\Vert  \bT_{12}\Vert_{L^{2}(\Omega_{\eta_1})}^2\dt'
 +
c
\int_0^t 
\Vert  \partial_{t'}\eta_{12}\Vert_{W^{1,2}(\omega )}^2 \dt'
 +
c\int_0^t 
\Vert  \eta_{12}\Vert_{W^{2,2}(\omega )}^2\dt'.
\end{aligned}
\end{equation}  
If we combine the estimates $K_1,\ldots,K_5$ and absorb the $\delta$-terms (by choosing $\delta$ very small) into the left-hand side, 
we obtain
\begin{equation}
\begin{aligned} \nonumber
 \Vert \bT_{12}(t) \Vert_{L^2(\Omega_{\eta_1})}^2
+
\int_0^t 
\Vert \bT_{12}\Vert_{L^2(\Omega_{\eta_1})}^2 \dt'
\lesssim
&\int_0^t \Vert \bu_1\Vert_{W^{3,2}(\Omega_{\eta_1})}\Vert \bT_{12}\Vert_{L^2(\Omega_{\eta_1})}^2
\dt'
+ 
\int_0^t
\Vert\nabx \bu_{12}\Vert_{L^{2}(\Omega_{\eta_1})}^2
\dt'
\\&+ 
\int_0^t
\big(\Vert  \partial_{t'}\eta_{12}\Vert_{W^{1,2}(\omega )}^2
+\Vert  \eta_{12} \Vert_{W^{2,2}(\omega )}^2
\big) \dt'
\end{aligned}
\end{equation} 
for all $t\in {I_*}$ so that by Gr\"onwall's lemma and the fact that $\bu_1$ is a strong solution,
\begin{equation}
\begin{aligned} \label{contrEst0}
\sup_{t\in I_*}\Vert \bT_{12}(t) \Vert_{L^2(\Omega_{\eta_1})}^2 
\lesssim
\int_{I_*}
\big(\Vert\nabx \bu_{12}\Vert_{L^{2}(\Omega_{\eta_1})}^2
+\Vert  \partial_{t}\eta_{12}\Vert_{W^{1,2}(\omega )}^2
+\Vert  \eta_{12} \Vert_{W^{2,2}(\omega )}^2
\big) \dt.
\end{aligned}
\end{equation} 
Now, following from Theorem \ref{thm:BM}, let consider two strong solutions $( \eta_i, \bu_i,  p_i )$, $i=1,2$  of  \eqref{divfreeAlone}--\eqref{interfaceAlone} with data $(\bff, g, \eta_0, \eta_\star, \bu_0, \underline{\bT}_i)$, respectively. For 
\begin{align*}
\underline{\bT}_{12}:=\underline{\bT}_1-\underline{\overline{\bT}}_2,
\quad
\bu_{12}=\bu_1-\overline{\bu}_2,  
\quad
\eta_{12}=\eta_1-\eta_2,
\end{align*}
where $\underline{\overline{\bT}}_2:=\underline{\bT}_2\circ\bm{\Psi}_{\eta_2-\eta_1}$, it follows from \cite[Remark 5.2]{BMSS} (c.f. \eqref{standEnergy}) that
\begin{align*}
\sup_{t\in {I_*}}
\Vert  \eta_{12} \Vert_{W^{2,2}(\omega )}^2
+
\int_{I_*}
\big(
\Vert  \nabx\bu_{12}\Vert_{L^{2}(\Omega_{\eta_1})}^2+
\Vert  \partial_t\eta_{12}\Vert_{W^{1,2}(\omega )}^2\big)
\dt
&\lesssim
\int_{I_*}
\Vert  \underline{\bT}_{12}\Vert_{L^{2}(\Omega_{\eta_1})}^2
\dt
\\&\lesssim
T_*
\Vert  \underline{\bT}_{12}\Vert_{Y_{\eta_1} }^2.
\end{align*}
Inserting into \eqref{contrEst0} then yields
\begin{equation}
\begin{aligned}
\label{contrEst1}
\Vert  \bT_{12}\Vert_{Y_{\eta_1}}^2
&\leq 
cT_*
\Vert  \underline{\bT}_{12}\Vert_{ Y_{\eta_1}}^2
.
\end{aligned}
\end{equation} 
If we choose  $T_*<\tfrac{1}{2c}$,
then we conclude from \eqref{contrEst1} that
\begin{equation}
\begin{aligned}
\nonumber
\Vert   \bT_{12} \Vert_{ Y_{\eta_1}}^2
&\leq
\tfrac{1}{2}
\Vert  \underline{\bT}_{12} \Vert_{Y_{\eta_1} }^2.
\end{aligned}
\end{equation} 
The existence of the desired fixed point now follows and finishes the proof of our main result, Theorem \ref{thm:MAIN}.
 
%
%

 \section*{Statements and Declarations} 
\subsection*{Funding}
 This work has been partly supported by Grant ME 6391/1-1 (543675748) by the German
Research Foundation (DFG).
\subsection*{Author Contribution}
The author wrote and reviewed the manuscript.
\subsection*{Conflict of Interest}
The author declares that they have no conflict of interest.
\subsection*{Data Availability Statement}
Data sharing is not applicable to this article as no datasets were generated
or analyzed during the current study.
\subsection*{Competing Interests}
The author have no competing interests to declare that are relevant to the content of this article.

\end{document}